\documentclass[12pt,reqno]{amsart}
\usepackage{hyperref}
\usepackage{amsfonts,mathrsfs,bbm,latexsym,rawfonts,amsmath,amssymb,amsthm,epsfig}
\usepackage{setspace, color}
\usepackage{fullpage}
\usepackage{amsaddr}   % 使用默认选项，地址会显示在作者下方
\usepackage{tikz}
\usetikzlibrary{arrows.meta, positioning}
\allowdisplaybreaks[4]

\newtheorem{thm}{Theorem}[section]

\newtheorem{rem}[thm]{Remark}
\newtheorem{lem}[thm]{Lemma}
\newtheorem{prop}[thm]{Proposition}
\newtheorem{defn}[thm]{Definition}

\newcounter{alphproblem}
\renewcommand{\thealphproblem}{\alph{alphproblem}}
\newenvironment{alphproblem}{
	\refstepcounter{alphproblem}
	\textbf{Problem \thealphproblem.}
	\ignorespaces
}
\numberwithin{equation}{section}

\newcommand{\al}{\alpha}

\newcommand{\ld}{\lambda}

\newcommand{\De}{\Delta}
\newcommand{\ep}{\varepsilon}

\newcommand{\si}{\sigma}
\newcommand{\om}{\omega}
\newcommand{\Om}{\Omega}

\newcommand{\Ga}{\Gamma}

\newcommand{\U}{\mathbb{S}}

\newcommand{\T}{\mathcal{T}}

\DeclareMathOperator{\tr}{tr}

\newcommand{\Real}{\mathbb{R}}

\newcommand{\norm}[1]{\Vert#1\Vert}

\def\<{\left\langle} \def\>{\right\rangle}
\def\({\left(} \def\){\right)}
\newcommand{\n}{\nabla}
\newcommand{\p}{\partial}

\subjclass[2020]{Primary 35G61, 35Q55, 35Q60, 58J35}

\keywords{The Schr\"{o}dinger flow, Landau-Lifshitz equation, The initial-Neumann boundary value problem, Local well-posedness}
\begin{document}
\title[The Schr\"odinger flow]{Local well-posedness of the Schr\"odinger flow into $\mathbb{S}^2$ with natural boundary conditions}
\thanks{*Corresponding Author}
\author[B. Chen]{\small Bo Chen}
\address{\small School of Mathematics, South China University of Technology, Guangzhou, 510640, China}
\email{cbmath@scut.edu.cn}
	
\author[Y.D. Wang]{\small Youde Wang*}
\address{\small 1. School of Mathematics and Information Sciences, Guangzhou University;
		2. Hua Loo-Keng Key Laboratory of Mathematics, Institute of Mathematics, AMSS, and School of
		Mathematical Sciences, UCAS, Beijing 100190, China.}
\email{wyd@math.ac.cn}
\date{\today}

%\tableofcontents
\begin{abstract}
In this paper, we develop a new approximation scheme to solve the local well-posedness problem for the Schr\"odinger flow into the standard unit 2-sphere $\mathbb{S}^2\subset\Real^3$ (i.e., the Landau-Lifshitz equation) with natural boundary conditions.
\end{abstract}

\maketitle

\section{Introduction}\label{intr}
In 1935, Landau and Lifshitz \cite{LL} derived the celebrated ferromagnetic chain equation, commonly known as the Landau-Lifshitz (LL) equation:
\[\p_tu =-u\times\De u.\]
They further formulated the initial-Neumann boundary value problem of LL equation, which is also called LL equation with natural boundary conditions, as follows
\begin{equation}\label{LL}
\begin{cases}
\p_tu =-u\times\De u,\quad\quad&\text{(x,t)}\in\Om\times \Real^+,\\[1ex]
\frac{\p u}{\p \nu}=0, &\text{(x,t)}\in\p\Om\times \Real^+,\\[1ex]
u(x,0)=u_0: \Om\to \U^2,
\end{cases}
\end{equation}
where $\Om$ is a bounded domain in Euclidean spaces $\Real^{m}$ with $m\leq 3$, $\nu$ is the outward unit normal vector of $\p \Om$.

In this paper, we study the local well-posedness of the problem \eqref{LL} under the necessary compatibility conditions (see~~Definition \ref{def-comp} in this paper) for the initial data $u_0$. This problem remains open and fundamentally challenging for an extended period, despite partial progress in our earlier work \cite{CW2,CW4}, where the existence and uniqueness of local regular solutions were established under certain sufficient boundary compatibility conditions for $u_0$.

\subsection{Definitions and backgrounds}\

\medskip
Let $\Om$ be a bounded domain in $\Real^m$ with $m\leq 3$. In physics, for a time-dependent map $u$ from $\Om$ into $\U^2$ (where $\U^2$ denotes the unit sphere in $\Real^3$), the Landau-Lifshitz (LL) equation
\begin{equation}\label{LL-eq0}
	\p_t u=-u\times \De u	
\end{equation}
was first proposed by Landau and Lifshitz \cite{LL} in 1935 as a phenomenological model for studying the dispersive behavior of magnetization in ferromagnetic materials. In 1955, Gilbert \cite{G} extended this model by introducing a dissipative term, leading to the Landau-Lifshitz-Gilbert (LLG) equation:
\begin{equation*}
\p_tu=\beta u\times \Delta u-\al u\times (u\times \Delta u),
\end{equation*}
where $\beta$ is a real number and $\al \geq 0$ is the Gilbert damping coefficient. Here ``$\times$" denotes the cross product in $\Real^{3}$ and $\De$ is the Laplace operator in $\Real^{3}$.

In fact, this equation is closely related to the material sciences \cite{CDMS}. Ferromagnetic materials have intrinsic magnetic orders (magnetization) and exhibit bistable structures, making them widely used in data storage devices. One prominent example is permalloy, a nickel-iron magnetic alloy that typically adopts a face-centered cubic phase. However, a real permalloy often displays an irregular polycrystalline structure \cite{Sch}. Its magnetization dynamics of such materials is modeled by the multi-scale LLG equation \cite{G, LL} with locally periodic material coefficients.

Specifically, consider a material on an open, bounded and connected domain $\Omega$. The magnetization $u: \Omega\subset\mathbb{R}^3 \to \mathbb{S}^2$ is described by the following multiscale LLG equation in dimensionless form:
$$\partial_tu-\alpha u\times\partial_t u= -(1+\alpha^2)u\times h(u),$$
where $0<\alpha<1$ is the Gilbert damping constant. The effective field $h(u)$ takes the form
$$h(u) = \mathrm{div}(a(x)\nabla u) - k(x)(u-(u\cdot e)e) + h_s(u) + h_e.$$
The exchange coefficient matrix denoted by $a(x) = (a_{ij}(x))$ ($i, j=1, 2, 3$) is assumed to be symmetric and satisfies
$$\lambda |\xi|^2 \leq \xi\cdot a\xi \leq \Lambda |\xi|^2$$
for any $\xi \in \mathbb{R}^3$ a.e. on $\Omega$ with $0 < \lambda < \Lambda$. The anisotropy coefficient $k(x)$ is a positive bounded scalar function, and the constant vector $e\in\mathbb{S}^2$ is the direction of the easy axis. The stray field $h_s(u) = -\nabla U$ and $U$ satisfies
$$U(x, t) =\int_{\Omega}\nabla N (x-y)\cdot u(y, t)dy,$$
where $N(x) = -\frac{1}{4\pi|x|}$ is the Newtonian potential and $h_e$ is the external magnetic field.

The LLG equation has since inspired numerous physically significant generalizations, including models incorporating spin current, spin-polarized transport, and magneto-elastic equation. For a comprehensive overview, we refer to \cite{Bo,GW,KKS,Sci,Nature} and the references therein.
\medskip

Mathematically, the negative sign ``$-$'' in equation \eqref{LL-eq0} does not affect our analysis and the main results. Therefore, for simplicity, we focus on the classical Schr\"{o}dinger flow into $\U^2$:
\[\p_tu =u\times\De u.\]
Geometrically, the cross product operator ``$u\times$" can be interpreted as a complex structure
\[J(u)=u\times: T_{u}\U^2\to T_{u}\U^2\]
in $\U^2$, which rotates the vectors in the tangent space of $\U^2$ anticlockwise by an angle of $\frac{\pi}{2}$ degrees. This allows us to rewrite the equation in an intrinsic form:
\[\p_tu = J(u)\tau(u),\]
where $$\tau(u)=\De u + |\p u|^2u$$ is the tension field of the map $u$, $\p$ denotes the Euclidean derivative.

From the perspective of infinite-dimensional symplectic geometry, Ding and Wang \cite{DW,DW1} introduced the Schr\"odinger flows for maps from a Riemannian manifold $(M,g)$ into a symplectic manifold $(N, J, h)$, which generalizes the LL equation \eqref{LL-eq0}, and was also independently developed by Terng and Uhlenbeck \cite{Uh1} from the viewpoint of integrable systems. Concretely, for a time-dependent map $u:M\times \Real^+\to N\hookrightarrow \Real^{K}$, the Schr\"odinger flow  is governed by
\[\p_t u=J(u)\tau(u).\]
Here, the tension field $\tau(u)$ has extrinsic expression
$$\tau(u)=\De_g u+A(u)(\p u,\p u).$$
where $A(\cdot, \cdot)$ is the second fundamental form of the embedding $N\hookrightarrow \Real^{K}$.

\subsection{Related works}\

\medskip
Due to their profound physical significance, the Schr\"odinger flow and closely related LL-type equations have attracted considerable attention from both physicists and mathematicians. During the past five decades, substantial progress has been made in understanding the well-posedness of weak and regular solutions to the Schr\"odinger flow in various geometric settings.

In 1986, P.L. Sulem, C. Sulem and C. Bardos \cite{SSB} established the existence of global weak solutions and local regular solutions to the Schr\"{o}dinger flow for maps from $\Real^n$ into $\U^2$, by employing the difference method. Later, Y.D. Wang \cite{W} extended this global existence result of weak solutions to the Schr\"odinger flow from a closed Riemannian manifold or a bounded domain in $\Real^n$ into $\U^2$, by using the complex structure approximation method. For the hyperbolic plane $\mathbb{H}^2$ as a target, A. Nahmod, J. Shatah, L. Vega and C.C. Zeng \cite{NSVZ} investigated the existence of global weak solutions to the Schr\"odinger flow defined on $\Real^2$. Further developments of weak solutions to a class of generalized Schr\"odinger flows and related equations can be found in \cite{CW0,JW1,JW2} and references therein. However, the existence of global weak solutions to the Schr\"odinger flow for maps between a Riemannian manifold and a K\"ahler manifold remains an open problem.

The local existence theory for the Schr\"odinger flow, initiated in \cite{SSB}, has been generalized in several directions. Ding and Wang \cite{DW} demonstrated the existence of local regular solutions to the Schr\"odinger flow from a closed Riemanian manifold or $\Real^n$ into a K\"ahler manifold by applying a parabolic geometric approximation equation and estimated some appropriate intrinsic geometric energy. Later, for low-regularity initial data, A.R. Nahmod, A. Stefanov and K. Uhlenbeck \cite{NSU} gained a near-optimal (but conditional) local well-posedness result for the Schr\"odinger map flow from $\Real^2$ into the sphere $X =\mathbb{S}^2$ or the hyperbolic space $X = \mathbb{H}^2$, by using the standard technique of Picard iteration in suitable function spaces of the Schr\"odinger equation. Both results include persistence of regularity, which means that the solution always stays as regular as the initial data (as measured in Sobolev norms), provided that it is within the time of existence guaranteed by the local existence theorem.

For one dimensional global existence for Schr\"odinger flow from either $\mathbb{S}^1$ or $\mathbb{R}^1$ into a K\"ahler manifold, we refer to  the works \cite{ZGT,PWW, RRS} and the recent preprint \cite{WZ}.  In higher dimensions $n\geq 2$, the global well-posedness result for the Schr\"odinger flow from $\mathbb{R}^n$ into $\U^2$ with small initial data has been extensively investigated by Bejanaru, Ionescu, Kenig and Tataru, see \cite{IK,BIK,B1,BIKT} for detailed results. Notably, the global well-posedness result for the Schr\"odinger flow for small data in the critical Sobolev spaces in dimensions $n\geq2$ was addressed in \cite{BIKT}, and Z. Li \cite{L1, L2} later extended these results to compact K\"ahler targets.

For the equivariant Schr\"odinger flows,  Bejenaru, Ionescu, Kenig, and Tataru \cite{BIKT1, BIKT2} proved global well-posedness and scattering in two key settings: for flows from $\Real^2$ with energy below ground state, and for finite-energy equivariant flows into
hyperbolic space $\mathbb{H}^2$. In addition, the global dynamics of the Schr\"odinger flow on $\Real^n$ or $\mathbb{H}^2$ near harmonic maps were studied in \cite{GKT,GNT,BT,BT10,LLOS,LZ3}.

In contrast, for dimensions $n\geq 2$, the Schr\"odinger flow with large initial data is known to develop singularities. Finite-time blow-up solutions near harmonic maps for 1-equivariant flows were constructed by Merle, Raphael and Rodnianski \cite{MRR} and Perelman \cite{GP}. Furthermore, self-similar finite-time blow-up solutions with locally bounded energy for the Schr\"odinger flow from $\mathbb{C}^n$ into $\mathbb{C}P^n$ were obtained in \cite{DTZ, GSZ,NSVZ}. The vortex-structured traveling wave solutions for the Schr\"odinger flow were constructed by F. Lin and J. Wei \cite{LW} and later by J. Wei and J. Yang \cite{WY}.

Despite these advances, the initial-boundary value problem for the Schr\"odinger flow in a manifold $M$ with boundary (where $\mbox{dim}(M)\geq 2$) remains largely unexplored. The study of such problems goes back to the foundational work of Landau and Lifshitz \cite{LL}, who introduced the initial-Neumann boundary value (INB) problem of the ferromagnetic chain equation (i.e. LL equation):
\begin{equation}\label{eq-LL}
\begin{cases}
\p_tu =u\times\De u,\quad\quad&\text{(x,t)}\in\Om\times \Real^+,\\[1ex]
\frac{\p u}{\p \nu}=0, &\text{(x,t)}\in\p\Om\times \Real^+,\\[1ex]
u(x,0)=u_0: \Om\to \U^2,
\end{cases}
\end{equation}
where $\Om$ is a bounded domain in Euclidean spaces $\Real^{m}$ with $m\leq  3$, and $u$ is a map from $\Om$ into the unit sphere $\U^2$. For decades, progress in establishing regular solutions to this problem  has been limited. In our previous work \cite{CW2}, we proved the existence and uniqueness of local strong solutions to the problem \eqref{eq-LL}, under the assumption that $u_0\in W^{3,2}(\Om,\U^2)$ satisfies the 0-th order compatibility condition $\frac{\p u_0}{\p \nu}|_{\p \Om}=0$.  The proof relied on a parabolic approximation scheme and the derivation of novel equivalent $W^{3,2}$-energy estimates for the approximation solutions.  Recently, in \cite{CW4}, we established the local existence of highly regular solutions to the problem \eqref{eq-LL} within Sobolev spaces. Specifically, we obtained the following results.

\begin{thm}\label{thm1}
Let $\Om$ be a smooth bounded domain in $\Real^3$. Suppose that $u_0\in W^{5,2}(\Om,\U^2)$, which satisfies the $1$-th order compatibility conditions:
\begin{align}
 \frac{\p u_0}{\p \nu}|_{\p \Om}=0\quad\quad\mbox{and}\quad\quad \n_\nu \tau(u_0)|_{\p\Om}=0,\label{1-order-cc}  
\end{align}
where $\n$ is the pull-back connection on ${u_0}^*(T\U^2)$. Then there exists a positive time $T_0$ depending only on $\norm{u_0}_{W^{5,2}(\Om)}$ such that problem \eqref{eq-LL} admits a unique solution $u$ on $[0,T_0]$, which satisfies
\[\p^i_tu\in L^\infty([0,T_0], W^{5-2i,2}(\Om))\]
for $i=0,1,2$.
\end{thm}

Additionally, we further proved higher regularities of the solution $u$ given in Theorem \ref{thm1}, by providing the following $k$-th order compatibility conditions for $u_0$ at the boundary, for any $k\geq 2$:
\begin{itemize}
\item[$\blacklozenge$]  \emph{For all $0\leq j\leq 2k$, there holds
\begin{equation*}
\frac{\p}{\p \nu}\p^ju_0|_{\p\Om}=0,
\end{equation*}
where $\p^ju_0=\(\frac{\p^ju_0}{\p x^{i_1}\cdots\p x^{i_j}}\)$ are all the $j$-th partial derivatives of $u_0$.}	
\end{itemize}

However, for the existence of local smooth solutions to the INB problem \eqref{eq-LL}, the necessary $k$-th order compatibility conditions of $u_0$ are significantly weaker. These conditions are intrinsically defined by the following:
\begin{itemize}
\item[$\bullet$] For all $0\leq j\leq k$, the initial data $u_0$ satisfies
\begin{equation}\label{com-cond-1}
\n_\nu v_j(0)|_{\p \Om}=0,
\end{equation}
\end{itemize}
where $v_j(0)=\n^j_t u|_{t=0}\in u^*_0(T\U^2)$, provided that $u$ is a regular solution to problem \eqref{eq-LL} on $[0,T]$, and $\n$ is the pull-back connection on $u^*(T\U^2)$. In particular, $v_1(0)=u_0\times \tau(u_0)$. Hence, condition $\n_\nu \tau(u_0)|_{\p \Om}=0$ (see \eqref{1-order-cc}) is equivalent to 
\[\n_\nu v_1(0)|_{\p \Om}=0.\]  
We refer to Proposition \ref{intrc-cd} for more details. 

\medskip
This raises the following open problem:

\medskip
\begin{alphproblem}\label{problem1}
\emph{Does the LL equation with natural boundary conditions (i.e., the INB problem \eqref{eq-LL}) admit a local smooth solution if and only if  the initial data satisfies the necessary $k$-th order compatibility conditions \eqref{com-cond-1} for each $k\in \mathbb{N}$?}
\end{alphproblem}

In general, if the target manifold is an arbitrary compact K\"ahler manifold $N$, the well-posedness for the INB problem to the Schr\"odinger flow:
\begin{equation}\label{S-eq1}
\begin{cases}
\p_tu =J(u)\tau(u),&\text{(x,t)}\in \Om\times [0,T],\\[1ex]
\frac{\p u}{\p \nu}=0, &\text{(x,t)}\in\p \Om\times [0,T],\\[1ex]
u(x,0)=u_0: \Om\to N,
\end{cases}
\end{equation}
remains unclear. This leads us to pose the following problem:

\medskip
\begin{alphproblem}\label{problem2}
\emph{Can we establish the existence of regular solutions (or even smooth solutions) to the problem \eqref{S-eq1} under the necessary compatibility conditions (see \eqref{com-cond1})?}
\end{alphproblem}
\subsection{Main results} In this paper, we provide a complete resolution of the above Problem \ref{problem1}. However, our approach is not directly applicable to the problem \ref{problem2}. Resolving this second problem \ref{problem2} requires the development of new methods to overcome several fundamental difficulties, and will be addressed in a forthcoming paper.

\medskip
Our main results of this paper can be presented as follows.

\begin{thm}\label{main-thm}
Let $\Om$ be a smooth bounded domain in $\Real^3$, $k\in \mathbb{N}$ such that $k\geq 1$. Suppose that the initial map $u_0\in W^{2k+3,2}(\Om)$ satisfies the necessary $k$-th order compatibility conditions \eqref{com-cond-1}, i.e.,
\[\n_\nu v_j(0)|_{\p \Om}=0\]
for $0\leq j\leq k$. Then there exists a positive constant $T_0$ depending only on $\norm{u_0}_{W^{5,2}}$ such that INB problem \eqref{eq-LL} admits a unique solution on $[0,T_0]$, which satisfies
\[\p_t^iu \in L^\infty([0,T_0], W^{2k+3-2i,2}(\Om)),\]
where $0\leq i\leq k+1$.

Furthermore, if $u_0\in C^\infty(\bar{\Om})$ satisfies the $k$-th order compatibility conditions \eqref{com-cond-1} for each $k\in \mathbb{N}$. Then the solution $u$ is smooth on $\bar{\Om}\times[0,T_0]$.
\end{thm}

\begin{rem}
The case $k=0$ in Theorem \ref{main-thm} has already been established in our previous work \cite{CW2}. Consequently, Theorem~\ref{main-thm} in the present paper, together with Theorem~1.1 in~\cite{CW2}, provides a complete resolution of the local well-posedness problem for the Landau-Lifshitz equation with natural boundary conditions.
\end{rem}

\medskip
The rest of this paper is organized as follows. In Section \ref{s: strategy}, we outline the key strategies and main ideas behind the proof of Theorem \ref{main-thm}. In Section \ref{s: pre} we will introduce some fundamental notation pertaining to manifolds and Sobolev spaces, along with the derivation of essential preliminary lemmas. Section \ref{s: Sob-inerp-inq} is devoted to establishing some Sobolev-interpolation inequalities and their corollaries. Section \ref{s: com-cond} formulates the compatibility conditions for the boundary data and provides an intrinsic geometric characterization of these conditions. In Section \ref{s: high-eq-LL}, we derive higher-order evolution equations for the LL equation and develop preliminary results necessary for the proof of Theorem \ref{main-thm}. The complete proof of Theorem \ref{main-thm} is presented in Sections \ref{s: local-exist-appro-problem} and \ref{s: proof-main-thm}, where we detail the approximation arguments and establish uniform energy estimates for approximate solutions. Finally, Appendix \ref{s: parabolic-eq} provides a local existence theory for a unified parabolic system with Neumann boundary conditions.

\section{Strategies in the proof of Theorem \ref{main-thm}}\label{s: strategy}

In this section, we outline our strategies and main ideas to prove the main Theorem \ref{main-thm}.

\subsection{New approximation schemes for LL equation}\

We try to establish the local existence of regular solutions to the INB problem \eqref{eq-LL}, where the initial data $u_0$ satisfies the necessary compatibility conditions \eqref{com-cond-1}. The first difficulty we face is how to design approximation equations for the LL equation that preserve the original compatibility conditions for the initial data $u_0$?

In our previous work \cite{CW4}, we considered the initial-Neumman boundary value problem for the following parabolic perturbed LL equation
\begin{equation}\label{para-SMF}
\begin{cases}
\p_tu =(\ep I+u\times )\tau(u),&\text{(x,t)}\in \Om\times \Real^+,\\[1ex]
\frac{\p u}{\p \nu}=0, &\text{(x,t)}\in\p \Om\times \Real^+,\\[1ex]
u(x,0)=u_0: \Om\to \U^2,
\end{cases}
\end{equation}
which only shares the same $1$-th order compatibility conditions with the INB problem \eqref{eq-LL}, i.e.,
\[\n_\nu u_0=0, \quad \n_\nu v_1(0)=0.\]
Here, $I$ denotes the identity map and $\tau(u)=\De u+|\p u|^2u$. Consequently, we used solutions of \eqref{para-SMF} to approximate a $W^{5,2}$-regular solution of \eqref{eq-LL}. To establish the existence of highly regular solutions, we now must incorporate the $k$-th ($k\geq 2$) order compatibility conditions \eqref{com-cond-1} for $u_0$.  This necessitates the construction of new approximation schemes.

\medskip
A key observation of this paper is that the $k$-th ($k\geq 2$) order compatibility conditions \eqref{com-cond-1} are equivalent to the $1$-th order compatibility conditions of the INB problem for the Schr\"odinger system satisfied by $U_k=(u,v_1, \cdots, v_{k-1})$. Here, $u$ solves the LL equation \eqref{eq-LL}, and $v_i=\n^i_t u$ ($1\leq i\leq k-1$) satisfies the following intrinsic problem (see equation \eqref{eq-v_k} in Section \ref{s: high-eq-LL}):
\begin{equation}\label{Neumann-vi}
\begin{cases}
\n_t \om=u\times(-\n^*\n \om+R^{\U^2}(u)(\om, \n_j u)\n_j u+R_i)\\[1ex]
\frac{\p \om}{\p \nu}|_{\p\Om}=0,\\[1ex]
\om=v_i(0): \Om\to \Real^K.
\end{cases}
\end{equation}
Here, $\omega$ is an unknown vector field in the pull-back bundle $u^*(T\U^2)$, the term $R_i$ depends only on $(u,v_1,\cdots,v_{i-1})$
(see Section \ref{s: high-eq-LL} for a precise formula of $R_i$) and $\n^*$ denotes the dual operator of $\n$. 

Moreover, the problem \eqref{Neumann-vi} and its parabolic perturbed problem are of the same $1$-th order compatibility conditions. We formulate this observation as the following lemma.

\begin{lem}\label{uniform-cc}
Let $k\in \mathbb{N}$ with $k\geq 2$. Assume that $u_0\in W^{2k+2,2}(\Om)$ satisfies the $1$-th order compatibility conditions \eqref{com-cond-1}. For any $\ep\in [0,1)$ and $1\leq i\leq k-1$, the 1-th order compatibility conditions for the intrinsically parabolic perturbed problem (see equation \eqref{eq-para-v_k} in Section \ref{s: high-eq-LL}) of \eqref{Neumann-vi}
\begin{equation}\label{para-Nuemann-eq-vi}
\begin{cases}
\n_t \om_\ep=(\ep I+u\times)(-\n^*\n \om_\ep+R^{\U^2}(u)(\om_\ep, \n_j u)\n_j u+R_i),\\[1ex]
\frac{\p \om_\ep}{\p \nu}|_{\p\Om}=0,\\[1ex]
\om_\ep(0)=v_i(0): \Om\to \Real^K
\end{cases}
\end{equation}
are equivalent to
\[\n_\nu v_i(0)|_{\p\Om}=0\quad \text{and}\quad \n_\nu v_{i+1}(0)|_{\p\Om}=0.\]
\end{lem}

\begin{proof}
For each $1\leq i\leq k-1$, since $\om_\ep$ is a section in $u^*(T\U^2)$, we have
\[\n_t \om_\ep|_{t=0}=\p_t \om_\ep|_{t=0}+\<v_1(0), v_i(0)\>u_0.\]
It follows that the first order compatibility conditions of the initial data $v_i(0)$ for the problem \eqref{para-Nuemann-eq-vi}, i.e.,
\[\frac{\p}{\p \nu} v_i(0)|_{\p\Om}=0\quad \text{and}\quad\frac{\p}{\p \nu}\p_t\om_\ep|_{\{t=0\}\times{\p\Om}}=0,\]
are equivalent to
\[\n_\nu v_i(0)|_{\p\Om}=0\quad \text{and}\quad \n_\nu\n_t\om_\ep|_{\{t=0\}\times{\p\Om}}=0.\]

On the other hand, since $v_{i+1}(0)=\n^{i+1}_t u|_{t=0}$ and $\n^i_t u$ satisfies the equation \eqref{Neumann-vi}, we obtain
\[v_{i+1}(0)=J(u_0)\(-\n^*\n v_i(0)+R^{\U^2}(u)(v_i(0), \n_j u_0)\n_j u_0+R_i|_{t=0}\),\]
where we denote $J(u_0)=u_0\times$. By the construction of parabolic perturbed equation \eqref{para-Nuemann-eq-vi}, we then conclude that
\begin{align*}
\n_t\om_\ep|_{t=0}=&(\ep I+J(u_0))\(-\n^*\n v_i(0)+R^{\U^2}(u)(v_i(0), \n_j u_0)\n_j u_0+R_i|_{t=0}\)\\
=&-(\ep I+J(u_0))J(u_0)v_{i+1}(0)\\
=&-\ep J(u_0)v_{i+1}(0)+v_{i+1}(0).	
\end{align*}

The integrability of the complex structure $J(u_0)$, i.e., $\n^{\U^2}J(u_0)=0$, implies
\[0=\n_\nu\n_t\om_\ep|_{\{t=0\}\times{\p\Om}}=-\ep J(u_0)\n_\nu v_{i+1}(0)|_{\p\Om}+\n_\nu v_{i+1}(0)|_{\p\Om},\]
which simplifies to
\[\n_\nu v_{i+1}(0)|_{\p\Om}=0.\]

Therefore, the proof is completed.
\end{proof}

This lemma allows us to use solutions of problem \eqref{para-Nuemann-eq-vi} to approximate $W^{5,2}$-solutions to \eqref{Neumann-vi} for each $1\leq i\leq k-1$, which yields higher regularity of solutions for \eqref{eq-LL} when the initial data satisfies the $k$-th order condition \eqref{com-cond-1}.

\medskip
Let $U_k=(u,v_1,\cdots, v_{k-1})$. The following diagram summarizes our intrinsic approximation scheme:
 
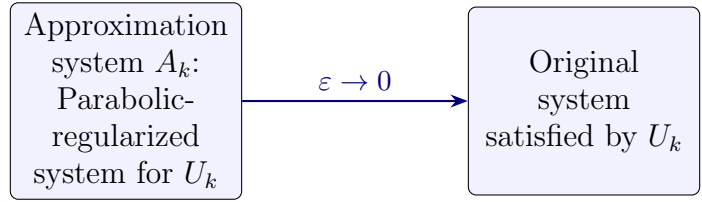
\begin{center}
 	\begin{tikzpicture}[
 		node distance=3cm,
 		block/.style={
 			draw,
 			rectangle,
 			minimum width=3cm,
 			minimum height=2.5cm,
 			align=center,
 			rounded corners=3pt,
 			fill=blue!5,
 			text width=2.8cm
 		},
 		arrow/.style={-Stealth, thick, blue!50!black},
 		label/.style={midway, above, font=\small}
 		]
 		
 		% Nodes
 		\node[block] (approximation) {
 			Parabolic perturbed \\
 			system for $U_k$};
 		
 		\node[block, right=of approximation] (original) {Original system \\
 			satisfied by $U_k$};
 		
 		% Arrow
 		\draw[arrow] (approximation) -- node[label] {$\varepsilon \to 0$} (original);
 		
 	\end{tikzpicture}
 \end{center}
The equivalence between the corresponding compatibility conditions is captured by the following:
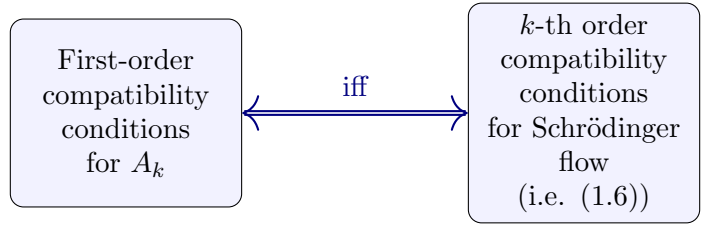
\begin{center}
\begin{tikzpicture}[
	node distance=3cm,
	block/.style={
		draw,
		rectangle,
		minimum width=3cm,
		minimum height=2.5cm,
		align=center,
		rounded corners=5pt,
		fill=blue!5,
		text width=2.8cm,
		font=\small
	},
	arrow/.style={<->, double, thick, blue!50!black},
	label/.style={midway, above, font=\small, yshift=2pt}
	]
	
	% Nodes
	\node[block] (first-order) {1-th order \\
		compatibility conditions \\
		for the parabolic perturbed \\
 			system of $U_k$};
	
	\node[block, right=of first-order] (kth-order) {$k$-th order \\
		compatibility conditions \\
		for Schr\"odinger flow \\
		(i.e., \eqref{com-cond-1})};
	
	% Arrow
	\draw[arrow] (first-order) -- node[label] {iff} (kth-order);
	
\end{tikzpicture}
\end{center}

\subsection{The outline of the proof of Theorem \ref{main-thm}}\

We proceed by induction to establish Theorem \ref{main-thm}. Let $u$ and $T_0$ be as given in Theorem \ref{thm1}. For each $k\geq 1$, we demonstrate that the solution $u$ satisfies the following property $\T_k$:
\begin{itemize}
\item Assume that $u_0\in W^{2k+3,2}(\Om)$ satisfies the $k$-th order compatibility conditions \eqref{com-cond-1}, then
\begin{align*}
\p_t^iu\in L^\infty([0,T_0],W^{2k+3-2i,2}(\Om)),
\end{align*}
for $0\leq i\leq k+1$.
\end{itemize}

Recall that property $\T_1$ has been established in \cite{CW4} (see Theorem \ref{thm1}). Assume that $\T_k$ with $k\geq 1$ holds, we establish $\T_{k+1}$ through the following steps:
\medskip

\textbf{1. Approximation equation}: Under the $(k+1)$-order compatibility conditions of $u_0$, we consider the following initial-Neumann boundary value problem satisfied by $v_{k}$ (i.e., the extrinsic version of the problem \eqref{Neumann-vi} with $i=k$):
\begin{equation}\label{Nuemann-eq-v_k1}
\begin{cases}
\p_t \om+\<v_1,\om\>u=u\times (\De \om+|\p u|^2\om+R_k)\\[1ex]
\frac{\p \om}{\p \nu}|_{\p\Om}=0,\\[1ex]
\om=v_k(0): \Om\to \Real^K.
\end{cases}
\end{equation}

Noting that $v_k\in L^\infty([0,T_0], W^{3,2}(\Om))$ is a strong solution to the above problem \eqref{Nuemann-eq-v_k1}. To enhance the regularity of $v_k$, we introduce a parabolic approximation scheme for problem \eqref{Nuemann-eq-v_k1} (i.e., the extrinsic version of problem \eqref{para-Nuemann-eq-vi} with $i=k$):
\begin{equation}\label{para-Nuemann-eq-v_k1}
\begin{cases}
\p_t \om_\ep+\<v_1,\om_\ep\>u=\ep (\De \om_\ep+2\<\p u, \p \om_\ep\>u+\<\De u, \om_\ep\>u+|\p u|^2\om_\ep+R_k)\\
\quad\quad\quad\quad\quad\quad \quad \quad \,\, + u\times(\De \om_\ep+|\p u|^2\om_\ep+R_k ),\\[1ex]
\frac{\p \om_\ep}{\p \nu}|_{\p\Om}=0,\\[1ex]
\om_\ep(0)=v_k(0): \Om\to \Real^K,
\end{cases}
\end{equation}
where $\om_\ep$ serves as an approximating sequence for $v_k$.

By Lemma \ref{uniform-cc}, both problems \eqref{Nuemann-eq-v_k1} and \eqref{para-Nuemann-eq-v_k1} share identical first order compatibility conditions, ensuring the validity of our approximation approach.
\medskip

\textbf{2. Improvement in regularity of $v_k$:} We take the following process to explain the improvement in regularity for $v_k$:
\begin{itemize}
\item[$(1)$] Obtain solutions $\om_\ep\in C^0([0,T_0], W^{5,2}(\Om))\cap L^2([0,T_0], W^{6,2}(\Om))$ to the problem \eqref{para-Nuemann-eq-v_k1} via Galerkin methods with appropriate energy estimates.
	
\item[$(2)$] Establish $\ep$-independent $W^{5,2}$-estimates for $\om_\ep$ and pass to the limit $\epsilon \to 0$ to obtain a solution $\om \in L^\infty([0,T_0], W^{5,2}(\Om))$ to the problem \eqref{Nuemann-eq-v_k1}. Uniqueness yields $v_k=\om$, completing the proof of $\mathcal{T}_{k+1}$.
\end{itemize}

\textbf{Technical challenges in energy estimates:} Establishing uniform $W^{5,2}$-energy estimates for approximate solutions $\om_\ep$ presents technical challenges, due to the limited space of admissible test functions (which vanish on boundary terms when integration by parts is applied).  Our strategy for overcoming this difficulty consists of the following key steps.

First, we establish the following innovative equivalent estimates for the Sobolev norms of $\om_\ep$ (see~~Lemmas \ref{equiv-es-om-ep} and \ref{equiv-W^{5,2}-es-om}):
\begin{equation}\label{equiv-es}
\begin{aligned}
	\norm{\om_\ep}^2_{W^{2,2}}\leq &C(\norm{\p_t \om_\ep}^2_{L^2}+\norm{\om_\ep}_{L^2}+1),\\
	\norm{\om_\ep}^2_{W^{4,2}}
	\leq &C(\norm{\p^2_t \om_\ep}^2_{L^2}+\norm{\p_t \om_\ep}^2_{L^2}+\norm{\om_\ep}_{L^2}+1),\\
	\norm{\om_\ep}^2_{W^{5,2}}
	\leq &C(\norm{(\p^2_t \om_\ep)^{\top}}^2_{W^{1,2}}+\norm{\p_t \om_\ep}^2_{L^2}+\norm{\om_\ep}_{L^2}+1),
\end{aligned}
\end{equation}
where
$$(\p^2_t \om_\ep)^{\top}=\p^2_t \om_\ep-\<\p^2_t \om_\ep, u\>u\in T_u\U^2$$
denotes the tangential component of $\p^2_t \om_\ep$ in $T_u\U^2$. These estimates reveal that establishing uniform $W^{5,2}$-energy bounds for $\om_\ep$ reduces to controlling the following composite energy:
\begin{align}
	\mathscr{E}(\om_\ep):=\norm{\om_\ep}^2_{L^2}+\norm{\p_t\om_\ep}^2_{L^2}+\norm{(\p^2_t \om_\ep)^\top}^2_{W^{1,2}}.\label{E}
\end{align}
Remarkably, this approach avoids the derivative loss that would occur in direct estimates for the $W^{5,2}$-norm of $\om_\ep$.

%Furthermore, the initial-Neumann boundary conditions $\frac{\p}{\p \nu}\p^i_{t}\om_\ep|_{\p \Om\times [0,T_0]} = 0$ (for $i = 0,1,2$) allow the use of $\p^i_t\om_\ep$ and $\De \p^i_t \om_\ep$ as test functions for the problem \eqref{para-Nuemann-eq-v_k1} during the process of energy estimates. This naturally leads us to consider the time-differentiated equations for $\p_t \om_\ep$ and $\p^2_t \om_\ep$, and then employ $\p^i_t\om_\ep$ and $\De \p^i_t \om_\ep$ ($i=0,1,2$) as test functions for these equations to establish uniform bounds for $\mathscr{E}(\om_\ep)$.

Second, we exploit the geometric structures of the target manifold $(\U^2, u\times)\hookrightarrow(\Real^3,\times)$ to eliminate certain high-order derivative terms when deriving uniform estimates for $\mathscr{E}(\om_\ep)$. These geometric structures include
\begin{itemize}
\item[$(1)$] Orthogonality: For any $X\in T_u \U^2$,
      \[\<X,\,u\>=0;\]
\item[$(2)$] Lie algebra structure: $(\Real^3,\times)$ is a Lie algebra;
\item[$(3)$] Triple product vanishing: For any $X_j\in T_u\U^2$~($j=1,2,3$),
	\[\<X_1\times X_2,\,X_3\>=0.\]
\end{itemize}

By systematically exploiting these geometric properties, all remaining terms in the energy estimates can be rigorously controlled through our equivalent Sobolev norm estimates for $\om_\ep$ (i.e. \eqref{equiv-es}). The complete technical details of this analysis are provided in Subsections  \ref{ss: uniform-en-es-1} and \ref{ss: uniform-en-es-2}.

\section{Preliminary}\label{s: pre}

\subsection{Geometric setup and Sobolev spaces}\

Let $\Om$ be a smooth bounded domain in $\Real^m$ with natural coordinates $\{x_1,\cdots, x_m\}$, and let $(N,J, h)$ be a compact K\"aher manifold.  By the Nash embedding theorem, we assume that $N$ is an isometrically embedded submanifold in a Euclidean space $\Real^K$. We denote by $A(\cdot, \cdot)$ the second fundamental form of $N$ in $\mathbb{R}^K$. For a smooth map $u:\Om\times [0,T]\to N$, let $\n$ denote the pull-back connection $u^*\n^N$, where $\n^N$ is the Levi-Civita connection induced by the metric $h$. For simplicity, we set
\[\n_t=\n_{\p_t} \quad\text{and}\quad \n_{j}=\n_{\p_{x_j}}\,\, \text{with}\,\, j=1,\cdots, m.\]

The tension field of $u$ is defined as
\[\tau(u)=\tr\n du=\n_j\n_ju=\De u+A(u)(\p u,\p u),\]
where $\p=(\p_1, \cdots, \p_m)$ is the Euclidean derivative, $\De$ is the Laplacian operator on $\Om$. Then, the Schr\"odinger flow is defined by the intrinsic equation
\[\n_t u=J(u)\tau(u),\]
while the parabolic perturbed Schr\"odinger flow satisfies
\[\n_t u=(\ep I+J(u)) \tau(u),\]
where $I$ denotes the identity map.

In the special case where $N$ is the standard unit sphere $\U^2$ in $\Real^3$ and the complex structure $J(u)=u\times: T_u\U^2\to T_u \U^2$, the LL equation and its parabolic perturbed equation (i.e. LLG equation) can be expressed intrinsically as
\[\n_t u=u\times \tau(u),\,\,\n_t u=(\ep I+u\times )\tau (u).\]
Here, the tension field is just
\[\tau(u)=\De u+|\p u|^2 u.\]

\medskip
Next, we introduce some notation on Sobolev spaces which are used in the subsequent context of this paper. For any $k\in \mathbb{N}$, we set
\[\p^k u=(\p^ku^1,\cdots, \p^ku^K).\]
The classical Sobolev space $W^{k,p}(\Om)$ is defined as the complete space of smooth function $u: \Om \to \Real^{K}$ under the following norm
\[\norm{u}_{W^{k,p}}=:\left\{\sum_{j=0}^{k}\int_{\Om}|\p^ju|^pdx\right\}^{\frac{1}{p}}.\]
Moreover, we also define the following
\[W^{k,2}(\Om,N)=\{u\in W^{k,2}(\Om)|\,\, u(x)\in N\,\,\text{for a.e.}\,\,x\in \Om\}.\]

More generally, let $(B, \norm{\cdot}_B)$ be a Banach space and $f:[0,T]\to B$ be a map. For any $p>0$ and $T>0$, we define
\[\norm{f}_{L^p([0,T], B)}:=\(\int_{0}^{T}\norm{f}^p_{B}dt\)^{\frac{1}{p}},\]
and
\[L^p([0,T],B):=\{f:[0,T]\to B:\norm{f}_{L^p([0,T],B)}<\infty\}.\]
In particular, we denote
\begin{align*}
	&L^{p}([0,T],W^{k,2}(\Om,N))\\
	=&\{u\in L^{p}([0,T],W^{k,2}(\Om)):u(x,t)\in N\,\,\text{for a.e. (x,t)}\in \Om\times[0,T]\},	
\end{align*}
where $p\geq 1$ and $k\in \mathbb{N}$.

\medskip
\subsection{Preliminary lemmas}\

This subsection collects several technical tools that will be essential for our subsequent analysis. We begin with a result on equivalent Sobolev norms under Neumann boundary conditions, see \cite{Weh} for a proof.

\begin{lem}\label{eq-norm}
Let $\Om$ be a bounded smooth domain in $\Real^{m}$ and $k\in\mathbb{N}$. Then, there exists a constant $C_{k,m}$ such that, for all $u\in W^{k+2,2}(\Om)$ with $\frac{\p u}{\p \nu}|_{\p\Om}=0$,
\begin{equation}\label{eq-n}
\norm{u}_{W^{k+2,2}(\Om)}\leq C_{k,m}(\norm{u}_{L^{2}(\Om)}+\norm{\De u}_{W^{k,2}(\Om)}).
\end{equation}	
\end{lem}
The above lemma implies that we can define the $W^{k+2,2}$-norm of $u$ as the following
$$\norm{u}_{W^{k+2,2}(\Om)}:=\norm{u}_{L^2(\Om)}+\norm{\De u}_{W^{k,2}(\Om)}.$$

\medskip
In order to establish uniform estimates and the convergence of solutions to the approximation equation constructed in subsequent sections, we also require the following Gronwall's inequality and the classical Aubin-Simon compactness results.

\begin{lem}\label{Gron-inq}
Let $y$ be a continuous function that is nonnegative on $[0,T_0]$. Assume that there exists a $y_0>0$ such that for $0\leq t\leq T_0$,
	\[y(t)\leq y_0+C\int_{0}^{t}y(s)ds.\]
Then we have
\[\sup_{0\leq t\leq T_0}y(t)\leq \exp(CT_0)y_0.\]
\end{lem}

A direct corollary of Lemma \ref{Gron-inq} is as follows.
\begin{lem}\label{ode}
Let $y:\Real^+\to \Real$ be a $W^{1,1}$ function such that
	\begin{equation*}
		\begin{cases}
			y^\prime\leq Cy,\\[1ex]
			y(0)\leq y_0.	
		\end{cases}
	\end{equation*}
Then we have
\[\sup_{0\leq t\leq T_0}y(t)\leq \exp(CT_0)y_0.\]
\end{lem}
%\begin{proof}
%Let
%$$w(t)=y_0+\int_{0}^{t}f(y(s))ds.$$
%It is easy to see that $w$ is a nondecreasing $C^1$ function, which satisfies
%\begin{equation*}
%\begin{cases}
%w^\prime=f(y(t))\leq f(w(t)),\\[1ex]
%w(0)=y_0\leq z_0.	
%\end{cases}
%\end{equation*}
%Here we have used the fact that $f$ is positive and nondecreasing. Then, the classical ODE comparison theorem tells us that
%\[w(t)\leq z(t)\]
%for any $t\in [0,T^*)$.
%
%Therefore, we get the desired result since $y(t)\leq w(t)$.
%\end{proof}

Let $1\leq p, r\leq \infty$, $X$ and $Y$ be two Banach spaces. For simplicity, we define
\[E_{p,r}=\left\{f\in L^{p}((0,T), X), \frac{d f}{dt}\in L^{r}((0,T), Y)\right\}.\]
In particular, if we take $X=W^{k+2,2}(\Om)$ and $Y=W^{k,2}(\Om)$, which are two Hilbert spaces, we set
\[E_{2,2}=\left\{f\in L^{2}((0,T),W^{k+2,2}(\Om)),\frac{\p f}{\p t}\in L^{2}((0,T), W^{k,2}(\Om))\right\}.\]
The following two embedding theorems can be found in \cite{BF,Sim}.

\begin{lem}[Theorem II.5.16 in \cite{BF}]\label{A-S}
Let $X\subset B\subset Y$ be Banach spaces, and $1\leq p,q,r\leq \infty$. Suppose that the embedding $B\hookrightarrow Y$ is continuous and that the embedding $X\hookrightarrow B$ is compact. Then the following properties are true.
\begin{itemize}
\item[$(1)$] If $p< \infty$, the embedding $E_{p,r}$ in $L^p((0,T), B)$ is compact.
\item[$(2)$] If $p< \infty$ and $p<q$, the embedding $E_{p,r}\cap L^q((0,T), B)$ in $L^s((0,T), B)$ is compact for all $1\leq s<q$.
\item[$(3)$] If $p=\infty$ and $r>1$, the embedding of $E_{p,r}$ in $C^0([0,T], B)$ is compact.
\end{itemize}
\end{lem}
\begin{lem}[Theorem II.5.14 in \cite{BF}]\label{C^0-em}
Let $k\in \mathbb{N}$, then the space $E_{2,2}$ is continuously embedded in $C^0([0,T], W^{k+1,2}(\Om))$.
\end{lem}

\section{Sobolev-interpolation inequalities}\label{s: Sob-inerp-inq}

In this section, we recall the classical Sobolev-interpolation inequalities, see \cite{DW} for a proof.
\begin{thm}\label{G-N2}
Let $\Om$ be a smooth bounded domain in $\Real^m$. Let $q,r$ be real numbers with $1\leq q,r\leq \infty$ and $j,m$ integers with $0\leq j<n$. Then there exists a constant $C$ depending only on $q,r,j,n,m$ and the geometry of $\Om$ such that for any function $f\in W^{n, r}(\Om)\cap L^q(\Om)$, we have
\begin{equation}\label{G-N-ineq1}
\norm{\p^jf}_{L^p(\Om)}\leq C\norm{f}^a_{W^{n,r}(\Om)}\norm{f}^{1-a}_{L^q(\Om)},
\end{equation}
where
\[\frac{1}{p}=\frac{j}{m}+a(\frac{1}{r}-\frac{n}{m})+(1-a)\frac{1}{q},\]
for all $a\in [\frac{j}{n},1]$. When $r=\frac{m}{n-j}\neq 1$, the inequality \eqref{G-N-ineq1} is not valid for $a=1$.
\end{thm}	
\begin{rem}
1. Inequality \eqref{G-N-ineq1} with $a=1$ gives the classical Sobolev inequality:
\begin{itemize}
\item[$(1)$] For $r<\frac{m}{n}$, we have
\begin{equation}\label{Soblov-inq}
\norm{f}_{L^p}\leq C\norm{f}_{W^{n,r}}
\end{equation}
for any $f\in W^{n,r}$, where the constant $p$ is given by
\[\frac{1}{p}=\frac{1}{r}-\frac{n}{m}.\]
\item[$(2)$] For $r>\frac{m}{n}$, we have
\begin{equation}\label{Soblov-inq1}
\norm{f}_{L^\infty}\leq C\norm{f}_{W^{n,r}}
\end{equation}
for any $f\in W^{n,r}$.
\end{itemize}	
2. Inequality \eqref{G-N-ineq1} with $a=\frac{j}{n}$ gives the following classical interpolation inequality:
\[\norm{\p^jf}_{L^p(\Om)}\leq C\norm{f}^{j/n}_{W^{n,r}(\Om)}\norm{f}^{(n-j)/n}_{L^q(\Om)}\]
for any $f\in W^{n,r}$, where
\[\frac{1}{p}=\frac{j}{n}\cdot\frac{1}{r}+(1-\frac{j}{n})\cdot\frac{1}{q}.\]
\end{rem}

For later arguments, we require the following direct corollary of Theorem \ref{G-N2} in the case where $m\leq 3$.
\begin{lem}\label{alg}
Let $\Om$ be a smooth bounded domain in $\Real^m$ with $m\leq 3$, $n_1\geq 0$, and $n_2\geq 2$. If $f\in W^{n_1,2}(\Om)$  and $g\in W^{n_2,2}(\Om)$, then $fg\in W^{l,2}(\Om)$ where $l=\min\{n_1,n_2\}$. Moreover, there exists a constant $C$ depending only on $n_1$ and $n_2$ such that we have
\[\norm{fg}_{W^{l,2}(\Om)}\leq C\norm{f}_{W^{n_1,2}} \norm{g}_{W^{n_2,2}}.\]	
Here we denote $W^{0,2}(\Om)=L^2(\Om)$ for the sake of convenience.
\end{lem}

Let $\Om$ be a smooth bounded domain in $\Real^3$, and let $u: \Om\times [0,T]\to \mathbb{S}^2$ be a sufficiently regular map.  Then, as a consequence of Lemma \ref{alg}, we obtain the following results on the equivalence of Sobolev norms for the extrinsic and intrinsic time derivatives of $u$.
\begin{lem}\label{equiv-es-p-t-u}
For any $k\geq 1$ and $l=0$ or $l=1$, the following two properties are equivalent.
\begin{itemize}
\item[$(1)$] For any $0\leq i\leq k$, $u$ satisfies
\[\p^i_t u\in L^\infty([0,T],\, W^{2k+l-2i}(\Om)).\]
\item[$(2)$] For any $0\leq i\leq k$, $u$ satisfies
\[v_i\in L^\infty([0,T],\, W^{2k+l-2i}(\Om)).\]
\end{itemize}
\end{lem}

\begin{proof}
We first prove that (1) yields (2). For any $1\leq i\leq k$, a simple computation shows
\begin{align*}
v_i=\p^i_t u+\sum_{\substack{j_1+\cdots+j_s=i,\\ 0\leq j_n\leq i-1, s\leq i+1}} \p^{j_1}_tu\#\cdots\# \p^{j_s}_tu,
\end{align*}
where $\#$ denotes the linear contraction. By property $(1)$, each $\p^{j_n}_t u$ satisfies
\[\p^{j_n}_tu \in L^\infty([0,T],\, W^{(2k+l-2i)+2,2}(\Om)),\]
since $j_n \leq i-1$ implies $(2k + l - 2j_n) \geq (2k + l - 2i) + 2$.
By employing Lemma \ref{alg}, we deduce
\[v_i\in L^\infty([0,T], W^{2k+l-2i,2}(\Om)).\]
	
 In contrast, for any $1\leq i\leq k$, we can write
\begin{align*}
\p^i_t u=v_i+\sum_{\substack{j_1+\cdots+j_s=i,\\ 0\leq j_n\leq i-1, s\leq i+1}} v_{j_1}\#\cdots\# v_{j_s}.
\end{align*}
Almost the same argument as the above one derives that
	\[\p^i_t u\in L^\infty([0,T],\, W^{2k+l-2i,2}(\Om)),\]
by providing that $u$ satisfies property $(2)$.
	
Therefore, the equivalence between (1) and (2) is thus established.
\end{proof}

\section{Compatibility conditions}\label{s: com-cond}
To establish the existence of regular solutions for the initial-Neumann boundary value problem associated with the Schr\"odinger flow, the imposition of appropriate compatibility conditions on the initial data becomes indispensable. In this section, we provide a rigorous formulation of these necessary compatibility conditions and further develop their intrinsic geometric characterizations.

\subsection{Compatibility conditions of the initial data}\
Let $(N,J,h)$ be a K\"ahler manifold that isometrically embeds in $\Real^K$ with the second fundamental form $A(\cdot, \cdot)$. Let $\Om$ be a bounded smooth domain in $\Real^3$. Suppose $u$ is a smooth solution to the initial-Neumann boundary value problem of the Schr\"odinger flow on $\bar{\Om}\times [0,T]$ for some $T>0$:
\begin{equation}\label{eq-ASMF}
\begin{cases}
\p_t u = J(u) \tau(u),\quad\quad&\text{(x,t)}\in \Om\times [0,T],\\[1ex]
\frac{\p u}{\p \nu}=0, &\text{(x,t)}\in\p \Om\times [0,T],\\[1ex]
u(x,0)=u_0: \Om\to N\hookrightarrow\Real^{K}.
\end{cases}
\end{equation}

Given the smoothness of $u$ and the Neumann boundary condition $\frac{\p u}{\p \nu}|_{\p\Om\times[0,T]}=0$, it follows by differentiation that
\[\frac{\p \p^k_t u}{\p \nu}|_{\p \Om\times [0,T]}=0\]
for any  $k\in \mathbb{N}$. Consequently, at $t=0$, we obtain the boundary conditions for $u_0$:
\[\frac{\p V_k}{\p \nu}|_{\p \Om}=0,\]
where we denote
\[V_k=\p^k_tu|_{t=0}.\]

In the special case where $(N, J) = (\mathbb{S}^2, u \times)$, we derive explicit recursive formulas for the compatibility conditions $V_k(u_0)$ (see also \cite{CJ,CW1,CW3}). The 0-order condition is simply $V_0 = u_0$, while for each integer $k \geq 1$, the higher-order terms satisfy
\begin{align*}
	V_k = \sum_{\substack{i+j=k-1 \ i,j \geq 0}} C_{k-1}^i (V_i \times \Delta V_j),
\end{align*}
where $C_{k-1}^i$ are combination numbers.
\medskip

Therefore, to guarantee the existence of sufficiently regular (or smooth) solutions to the Schr\"odinger flow \eqref{eq-ASMF}, the initial data $u_0$ must satisfy certain compatibility conditions at the boundary $\p \Om$. We now provide their rigorous mathematical formulation.
\begin{defn}\label{def-comp}
Let $k\in \mathbb{N}$, $u_0\in W^{2k+2,2}(\Om,N)$. We say $u_0$ satisfies the compatibility condition of order $k$, if for any $j\in\{0,1,\dots,k\}$,
	\begin{equation}\label{com-cond}
		\frac{\p V_j}{\p \nu}|_{\p \Om}=0.
	\end{equation}
\end{defn}

An intrinsic formulation of these compatibility conditions can be given in terms of covariant derivatives with respect to $t$. Define
$$v_k=\n^k_t u,$$
and set
$$v_k(0)=v_k|_{t=0}\in \Ga(u^*_0(TN)).$$
In particular, $v_0(0)=u_0$ and $v_1(0)=J(u_0)\tau(u_0)$.

Then the compatibility conditions defined in \eqref{com-cond} admit the following equivalent characterization.
\begin{prop}\label{intrc-cd}
For $k\in \mathbb{N}$, $u_0\in W^{2k+2,2}(\Om,N)$, the following are equivalent
\begin{enumerate}
\item $u_0$ satisfies the $k$-th order compatibility conditions (Definition \ref{def-comp});

\item For any $j\in\{0,1,\dots,k\}$,
\begin{equation}\label{com-cond1}
\left.\n_\nu v_j(0)\right|_{\partial \Omega} = 0.
\end{equation}
\end{enumerate}
\end{prop}

\begin{proof}
We establish the necessity by induction on $k$. Since $V_1=v_1(0)$, the assumption $\frac{\p V_1}{\p \nu}|_{\p\Om}=0$ implies
\[\n_\nu v_1(0)|_{\p\Om}=\frac{\p v_1(0)}{\p \nu}|_{\p\Om}+A(u_0)(\frac{\p u_0}{\p \nu}|_{\p\Om}, v_1(0))=0.\]
Next, we assume that the result holds for all $1\leq l\leq k-1$. For $l=k\geq 2$, we decompose $v_k(0)$ as
\begin{align*}
v_k(0)=V_k+\sum_{\si}B_{\si(k)}(u_0)(V_{a_1}, \cdots, V_{a_s})
\end{align*}
where the sum ranges over all partitions $\si(k)=(a_1, \cdots, a_s)$ of $k$ with  $1\leq a_i\leq k-1$ and $a_1+\cdots+a_s=k$, and each $B_{\si(k)}$ is a multi-linear vector-valued function in $\Real^K$. For details, we refer to page 1451 in \cite{DW}.
	
Applying the inductive hypothesis and boundary conditions, we compute
\begin{align*}
\n_\nu v_k(0)|_{\p\Om}=&\frac{\p v_k(0)}{\p \nu}|_{\p\Om}+A(u_0)(\frac{\p u_0}{\p \nu},v_k(0))|_{\p\Om}\\
=&\frac{\p V_k}{\p \nu}|_{\p\Om}+\sum_{\si}\p B_{\si(k)}(u_0)(\frac{\p u_0}{\p \nu}|_{\p \Om}, V_{a_1}, \cdots, V_{a_s})\\
&+\sum_{\si}B_{\si(k)}(u_0)\sum_{i}(V_{a_1}, \cdots,\frac{\p V_{a_i}}{\p \nu}|_{\p\Om},\cdots, V_{a_s})\\
=&0.
\end{align*}
This completes the induction.
	
The converse follows by reversing the argument, so we omit it.
\end{proof}

\subsection{Another compatibility conditions}\

We now establish the following lemma concerning compatibility conditions, which plays an essential role in eliminating boundary terms during the energy estimates in subsequent sections.
\begin{lem}\label{comp-cond}
Let $u: \Om\times [0,T]\to \Real$ be a function such that $\p^i_t u\in L^2([0,T], W^{4-2i,2}(\Om))$ for $i=0,1,2$. Assume that in the weak sense $\frac{\p u} {\p \nu}|_{\p\Om\times [0,T]}=0$, namely
\[\int_{0}^{T}\int_{\Om}\<\De f, \phi\>dxdt=-\int_{0}^{T}\int_{\Om}\<\p f, \p \phi\>dxdt\]
for any $\phi\in C^\infty(\Om\times [0,T])$, then
\[\frac{\p } {\p \nu}\p_t u|_{\p\Om\times [0,T]}=0.\]
\end{lem}
\begin{proof}
For any $\phi\in C^\infty(\bar{\Om}\times [0,T])$, integration by parts yields
\begin{equation}\label{comp}
\int_{0}^{T}\int_{\Om}\<\De u, \p_t \phi\>dxdt=-\int_{0}^{T}\int_{\Om}\<\p u, \p_t\p \phi\>dxdt,
\end{equation}
which holds by virtue of the Neumann boundary condition
$$\frac{\p u}{\p \nu}|_{\p\Om\times[0,T]}=0.$$

We analyze both sides of \eqref{comp} separately. For the left-hand side, we have the following.
\begin{align*}
\mbox{LHS of \eqref{comp}} =&-\int_{0}^{T}\int_{\Om}\<\p_t\De u, \phi\>dxdt+\int_{\Om}\<\De u, \phi\>dx(T)\\
&-\int_{\Om}\<\De u, \phi\>dx(0)\\
=&-\int_{0}^{T}\int_{\Om}\<\p_t\De u, \phi\>dxdt-\int_{\Om}\<\p u, \p\phi\>dx(T)\\
&+\int_{\Om}\<\p u, \p\phi\>dx(0).
\end{align*}
For the right-hand side, we have
\begin{align*}
\mbox{RHS of \eqref{comp}} =&-\int_{0}^{T}\int_{\Om}\<\p u, \p_t\p \phi\>dxdt\\
=&\int_{0}^{T}\int_{\Om}\<\p \p_tu, \p \phi\>dxdt-\int_{\Om}\<\p u, \p\phi\>dx(T)\\
&+\int_{\Om}\<\p u, \p\phi\>dx(0).
\end{align*}

Equating both sides and canceling boundary terms gives
\[\int_{0}^{T}\int_{\Om}\<\De \p_tu, \phi\>dxdt=-\int_{0}^{T}\int_{\Om}\<\p \p_tu, \p\phi\>dxdt,\]
which establishes the weak Neumann condition for $\p_t u$.  Here we have used Lemma \ref{C^0-em} to give
\[u\in C^0([0,T],W^{3,2}(\Om)).\]
Thus, if we take $\phi(x,t)=\eta(t)f(x)$, then
\begin{align*}
\int_{0}^{T}\(\int_{\Om}\<\De \p_t u, f\>dx+\int_{\Om}\<\p \p_t u, \p f\>dx\)\eta(t)dt=0.
\end{align*}
This implies
\[\int_{\Om}\<\De \p_t u, f\>dx=-\int_{\Om}\<\p \p_t u, \p f\>dx\]
for any $t\in [0,T]$.
\end{proof}
\begin{rem}
The test functions in Lemma \ref{comp-cond} can be taken from  $L^2([0,T], W^{1,2}(\Om))$ due to the density of $C^\infty(\Om\times[0, T])$ in this space.
\end{rem}

\section{Higher order evolution equations of the LL equation}\label{s: high-eq-LL}

In this section, we derive higher order evolution equations associated with the Landau-Lifshitz (LL)  equation, and then establish the necessary preliminary estimates for the proof of our main Theorem \ref{main-thm}. The complete proof will be presented in Sections \ref{s: local-exist-appro-problem} and \ref{s: proof-main-thm}.

\subsection{Evolution equations for higher time derivatives $v_k = \nabla^k_t u$}\

Let $J(u)=u\times: T_u\U^2\to T_u\U^2$ denote the standard complex structure on $\U^2$. Recall that LL equation admits the following intrinsic formulation
\[\p_t u=J(u)\n_j\n_ju.\]

Taking advantage of the constant curvature property $\nabla R^{\mathbb{S}^2} \equiv 0$ of the sphere, we compute the $k$-th order time derivative evolution:
\begin{align*}
\n_t \n^k_t u=&J(u)(\n^k_t \n_j\n_j u)\\
=&J(u)(\n_j\n_j\n^k_t u+R^{\mathbb{S}^2}(\n^k_t u, \n_ju)\n_ju+R_k)
\end{align*}
where the term $R_k$ (vanishing when $k=1$) is a section of $u^*T\mathbb{S}^2$ given by
\begin{align*}
R_k=&\sum_{\substack{a_1+\cdots, a_s+i+j=k,\\ 0\leq a_l,i,j\leq k-1}}v_{a_1}\#\cdots\#v_{a_s}\# \n v_i\#\n v_j\\
&+\sum_{\substack{b_1+b_2+b_3=k+1,\\ 1\leq b_l\leq k-1}}v_{b_1}\#J(u)v_{b_2}\#v_{b_3},
\end{align*}
with $\#$ representing the linear contraction. This establishes the intrinsic evolution equation for $v_k = \n^k_t u$:
\begin{equation}\label{eq-v_k}
\n_t \om=u\times(-\n^*\n \om+R^{\mathbb{S}^2}(\om, \n_ju)\n_ju+R_k),
\end{equation}
where $\n^*$ denotes the dual operator of $\n$, and hence $-\n^*\n \om=\n_i\n_i \om$. The corresponding parabolic regularization takes the form:
\begin{equation}\label{eq-para-v_k}
\n_t \om_\ep=(\ep I+u\times)(-\n^*\n \om_\ep+R^{\mathbb{S}^2}(\om_\ep, \n_ju)\n_ju+R_k).
\end{equation}

For the purpose of PDE analysis, we derive the extrinsic formulations of equations \eqref{eq-v_k} and \eqref{eq-para-v_k}.  By viewing $\om_\ep\in u^*(T\mathbb{S}^2)$ as a vector-valued function into $\Real^3$, we obtain the following extrinsic formula:
\begin{equation}\label{formula-laplace-and-curvature}
\begin{aligned}
	-\n^*\n \om_\ep=&\De \om_\ep+\p_i(\<\p_i u,\om_\ep\>u)+\<\p u,\p \om_\ep\> u\\
	=&\De \om_\ep+2\<\p u,\p \om_\ep\>u+\<\De u, \om_\ep\>u+\<\p_j u,\om_\ep\>\p_j u,\\
	R^{\mathbb{S}^2}(\om_\ep, \n_ju)\n_ju=&|\p u|^2 \om_\ep-\<\om_\ep,\p_j u\>\p_j u.
\end{aligned}
\end{equation}

Therefore, the extrinsic formulation of \eqref{eq-para-v_k} becomes
\begin{equation}\label{ex-eq-para-vk}
\begin{aligned}
\p_t \om_\ep+\<\om_\ep, v_1\>u=&\ep(\De \om_\ep+2\<\p u, \p \om_\ep\>u+\<\De u, \om_\ep\>u+|\p u|^2\om_\ep+R_k)\\
&+u\times (\De \om_\ep+|\p u|^2\om_\ep+R_k),
\end{aligned}
\end{equation}
where $R_k$ expressed extrinsically as
\begin{align*}
R_k=&\sum_{\substack{a_1+\cdots, a_s+i+j=k,\\ 0\leq a_l,i,j\leq k-1}}v_{a_1}\#\cdots\#v_{a_s}\# \p v_i\#\p v_j\\
&+\sum_{\substack{b_1+b_2+b_3=k+1,\\ 1\leq b_l\leq k-1}}v_{b_1}\#(u\times v_{b_2})\#v_{b_3}.
\end{align*}

\subsection{Proof Strategy for the main theorem}\

Next, incorporating higher-order compatibility conditions \eqref{com-cond-1} for the initial data $u_0$, we employ an inductive argument to establish the main theorem \ref{main-thm}. Specifically, for each $k\geq 1$, we prove the following property $\T_k$ recursively:
\begin{itemize}
\item Assume that $u_0\in W^{2k+3,2}(\Om)$ satisfies the $k$-order compatibility conditions \eqref{com-cond-1}, then
\begin{align}
\p^i_tu\in L^\infty([0,T_0],W^{2k+3-2i,2}(\Om)),\label{es-case-k}
\end{align}
for $0\leq i\leq k+1$.
\end{itemize}
Recall that property $\T_1$ was established in our previous work \cite{CW4}. Assuming now that $\T_k$ holds, we proceed to prove $\T_{k+1}$.

To proceed, under the $(k+1)$-th order compatibility conditions, we consider the initial-Neumann boundary value problem:
\begin{equation}\label{Nuemann-eq-v_k}
\begin{cases}
\n_t \om=u\times(-\n^*\n \om+R^{\U^2}(\om, \n_j u)\n_ju+R_k),\\[1ex]
\frac{\p \om}{\p \nu}|_{\p\Om}=0,\\[1ex]
\om=v_k(0): \Om\to \Real^K.
\end{cases}
\end{equation}
Note that $v_k \in L^\infty([0,T_0], W^{3,2})$ is a strong solution to this problem.

To enhance the regularity of $v_k$, we introduce a regularization scheme by considering the following parabolic approximation.
\begin{equation}\label{para-Nuemann-eq-v_k0}
\begin{cases}
\n_t \om_\ep=(\ep I+u\times)(-\n^*\n \om_\ep+R^{\U^2}(\om_\ep, \n_j u)\n_ju+R_k),\\[1ex]
\frac{\p \om_\ep}{\p \nu}|_{\p\Om}=0,\\[1ex]
\om=v_k(0): \Om\to \Real^K.
\end{cases}
\end{equation}

Since $u_0$ satisfies the $(k+1)$-th order compatibility conditions \eqref{com-cond-1}, we have
\begin{align}
\n_\nu v_k(0)|_{\p \Om}=0,\quad \n_\nu v_{k+1}(0)|_{\p \Om}=0.\label{1-order-com-cond-v_k}
\end{align}
Crucially, because the complex structure $J(u)=u\times$ is integrable (i.e. $\n J(u)=0$) and the orthogonality of the terms on the right-hand side of \eqref{eq-para-v_k}, the compatibility conditions for the initial data $v_k(0)$ in both problems \eqref{Nuemann-eq-v_k} and \eqref{para-Nuemann-eq-v_k0} coincide and reduce to \eqref{1-order-com-cond-v_k} (also see Lemma \ref{uniform-cc}). This ensures the validity of the parabolic approximation scheme.

The proof of $\T_{k+1}$ proceeds via the following two steps:
\begin{itemize}
	\item[$(1)$] We show the existence of $W^{5,2}$ regular solutions $\om_\ep$ to the approximate problem \eqref{para-Nuemann-eq-v_k} via Galerkin methods, and
	derive key a priori estimates. Details are provided in Section \ref{s: local-exist-appro-problem}.
	\item[$(2)$] We establish uniform $W^{5,2}$-estimates for $\om_\ep$ with respect to $\ep\in (0,1)$, and obtain $v_k\in L^\infty([0,T_0], W^{5,2}(\Om))$ by passing to the limit for the sequence of approximate solutions $\{\om_\ep\}$ as $\ep\to 0$, which completing the proof of the property $\T_{k+1}$. See Section \ref{s: proof-main-thm} for details.
\end{itemize}

\subsection{Estimates for nonhomogeneous term $R_k$}\

For later application, we estimate the nonhomogeneous term $R_k$ in equation \eqref{ex-eq-para-vk}, provided that $u$ satisfies property $\T_k$.

\begin{lem}\label{es-R}
Assume that the solution $u$ satisfies the property $\T_k$ with $k\geq 1$, then for $0\leq i\leq 2$, we have
\begin{align}
\p_t^iR_k\in L^\infty ([0,T_0], W^{4-2i,2}(\Om)).\label{es-R1}
\end{align}

\end{lem}

\begin{proof}
For simplicity, we decompose $R_k$ as $R_k=F^1_k+F^2_k$, where
\begin{align*}
F^1_k=&\sum_{\substack{a_1+\cdots, a_s+i+j=k,\\ 0\leq a_l,i,j\leq k-1}}v_{a_1}\#\cdots\#v_{a_s}\# \p v_i\#\p v_j,\\
F^2_k=&\sum_{\substack{b_1+b_2+b_3=k+1,\\ 1\leq b_l\leq k-1}}v_{b_1}\#(u\times v_{b_2})\#v_{b_3}.
\end{align*}
We provide only a detailed analysis for the estimates of $F_k^1$, as those of $F^2_k$ follow analogously. Without loss of generality, we assume that $k\geq 2$. We derive from the property $\T_k$ (i.e., \eqref{es-case-k}) and Lemma \ref{equiv-es-p-t-u} that
	\[v_j\in L^\infty([0,T_0], W^{5,2}(\Om))\]
for $0\leq j\leq k-1$, since $2k+3-2j\geq 5$. So, we can apply Lemma \ref{alg} to show
$$F_k^1\in L^\infty ([0,T_0], W^{4,2}(\Om)).$$
	
Now we consider the estimates of $\p_t F_k^1$ and $\p^2_t F_k^1$. The time derivatives of $F^1_k$ can be decomposed as
\begin{align*}
\p_t F_k^1=&\sum_{\substack{a_1+\cdots, a_s+i+j=k+1,\\ 0\leq a_l,i,j\leq k}} v_{a_1}\#\cdots\# v_{a_s}\# \p v_{i}\# \p v_{j},\\
\p^2_t F_k^1=&\sum_{\substack{a_1+\cdots, a_s+i+j=k+2,\\ 0\leq a_l,i,j\leq k+1}} v_{a_1}\#\cdots\# v_{a_s}\# \p v_{i}\# \p v_{j}.
\end{align*}
From property $\T_k$ (i.e., \eqref{es-case-k}), we have
	\[\p v_k \in L^\infty([0,T_0], W^{2,2}(\Om)) \quad\mbox{and}\quad \p v_{k+1} \in L^\infty([0,T_0], L^2(\Om)).\]
Then it follows from Lemma \ref{alg} that
	\[\p_t^iF_k^1\in L^\infty ([0,T_0], W^{4-2i,2}(\Om))\]
for $i=1,2$.
\end{proof}

\section{Local existence of the approximate problem}\label{s: local-exist-appro-problem}

In this section, we establish the existence of locally regular solutions of the approximate problem
\begin{equation}\label{para-Nuemann-eq-v_k}
\begin{cases}
\n_t \om_\ep=(\ep I+u\times)(-\n^*\n \om_\ep+R^{\mathbb{S}^2}(\om_\ep, \n_ju)\n_ju+R_k),\\[1ex]
\frac{\p \om_\ep}{\p \nu}|_{\p\Om}=0,\\[1ex]
\om_\ep(0)=v_k(0): \Om\to u^*(T\U^2),
\end{cases}
\end{equation}
and then obtain some a priori estimates for the approximate solutions $\om_\ep$.

Assume that $u_0\in W^{2(k+1)+3,2}(\Om)$ satisfies the $(k+1)$-order compatibility conditions given in \eqref{com-cond-1}, and that the property $\T_k$ holds for $k\geq 1$, namely the solution $u$ satisfies estimates \eqref{es-case-k}. Then we can derive from Lemma \ref{equiv-es-p-t-u} that
\begin{align}
v_i\in L^\infty([0,T_0],W^{2k+3-2i,2}(\Om)),\label{es-T-k}
\end{align}
for $0\leq i\leq k+1$. Applying Lemma \ref{C^0-em}, we further deduce that
\[v_i\in C^0([0,T_0],W^{2k+2-2i,2}(\Om))\]
for $0\leq i\leq k$.

On the other hand, from equation \eqref{eq-v_k}, we express $v_{k+1}$ as
\[v_{k+1}=u\times (\De v_k+|\p u|^2v_k+R_k),\]
which implies that $v_{k+1}\in C^0([0,T_0], L^2(\Om))$. Combining these results, we conclude that for all $0\leq i\leq k+1$,
\begin{align}
	v_i\in C^0([0,T_0],W^{2k+2-2i,2}(\Om))\label{es-T-k1}.
\end{align}

\subsection{Local existence of approximation problem}\

For the purpose of PDE analysis, we consider the extrinsic formulation of the above problem \eqref{para-Nuemann-eq-v_k}:
\begin{equation}\label{para-Nuemann-eq-v_1}
\begin{cases}
\p_t \om_\ep+\<\om_\ep, v_1\>u=\ep(\De \om_\ep+2\<\p u, \p \om_\ep\>u+\<\De u, \om_\ep\>u+|\p u|^2\om_\ep)\\
\quad\quad\quad\quad\quad\quad\quad\quad+u\times (\De \om_\ep+|\p u|^2\om_\ep)+(\ep I+u\times)R_k,\\[1ex]
\frac{\p \om_\ep}{\p \nu}|_{\p\Om}=0,\\[1ex]
\om_\ep(0)=v_k(0): \Om\to \Real^K.
\end{cases}
\end{equation}
Here, $\p$ denotes the derivative with respect to the coordinate $x$.

Our main result in this subsection is as follows.
\begin{thm}\label{W^{4,2}-solu-om-ep}
Assume that $u_0\in W^{2(k+1)+3,2}(\Om)$ satisfies the $(k+1)$-order compatibility conditions \eqref{com-cond-1}. Let $u$ be a regular solution to \eqref{eq-LL}, which satisfies the property $\T_k$ (i.e. estimates \eqref{es-case-k}). Then the problem \eqref{para-Nuemann-eq-v_1} admits a unique solution $\om_\ep$ on $[0,T_0]$ such that
\begin{align}
\p^i_t\om_\ep \in C^0([0,T_0], W^{4-2i,2}(\Om))\cap L^2([0,T_0], W^{5-2i,2}(\Om))\label{W^{4,2}-es-om-ep}
\end{align}
for $i=0,1,2$.
\end{thm}

\begin{proof}
The proof is divided into three steps.
\medskip

\noindent\emph{Step 1: $W^{3,2}$-regular solutions to problem \eqref{para-Nuemann-eq-v_1}.}\

Let
$$f_1=2\p u\otimes u,\quad f_2=\De u\otimes u+|\p u|^2I\quad \mbox{and} \quad f_3=(\ep I+u\times)R_k.$$
Applying the estimates \eqref{es-case-k} and \eqref{es-T-k} with $k\geq 1$, we obtain
\[f_i\in L^\infty([0,T_0], W^{3,2}(\Om))\quad\text{and}\quad \p_tf_i\in L^\infty([0,T_0], W^{1,2}(\Om)),\,\, i=1,2.\]
Furthermore, by Lemma \ref{es-R}, the function $f_3$ satisfies
\[\p^j_t f_3\in L^\infty([0,T_0], W^{4-2j,2}(\Om)),\,\,j=0,1,2.\]
Consequently, Lemma \ref{C^0-em} yields the following improved regularity
\[f_i \in C^0([0,T_0], W^{2,2}(\Om)),\,\, i=1,2,\]
and
\[\p^j_t f_3\in C^0([0,T_0], W^{3-2j,2}(\Om)),\,\, j=0,1.\]

Since $u_0\in W^{2(k+1)+3,2}(\Om)$ satisfies the $(k+1)$-order compatibility conditions \eqref{com-cond-1}, we have $v_i(0)\in W^{2(k+1-2i)+3,2}(\Om)$ and
\begin{align}
\frac{\p v_i(0)}{\p \nu}|_{\p \Om}=\n_\nu v_i(0)|_{\p\Om}-\<v_i(0),\n_\nu u_0\>u_0|_{\p\Om}=0.\label{comp-v-k}
\end{align}
for $1\leq i\leq k+1$. In particular, we have the special case:
\[v_k(0)\in W^{5,2}(\Om),\quad \frac{\p v_k(0)}{\p \nu}|_{\p \Om}=0.\]

Then, applying Theorem \ref{thm-h1} , we conclude that problem \eqref{para-Nuemann-eq-v_1} admits a unique solution $\om_\ep$ in $[0,T_0]$ with the following regularity properties:
\begin{align}
\p^i_t\om_\ep\in C^0([0,T_0], W^{3-2i,2}(\Om))\cap L^2([0, T_0],W^{4-2i}(\Om))\label{W^{3,2}-es-om}
\end{align}
for $i=0,1$. Additionally, we have $\p^2_t \om_\ep\in L^2([0,T_0], L^2(\Om))$.

\medskip
\noindent\emph{Step 2: $W^{2,2}$-regularity of $\p_t\om_\ep$. }\

To improve the regularity of $\om_\ep$, we consider the equation satisfied by $\p_t \om_\ep$:
\begin{equation}\label{para-Nuemann-eq-p-t-om}
\begin{cases}
\p_t \al+\<\al, v_1\>u=\ep(\De \al+2\<\p u, \p \al\>u+\<\De u, \al\>u+|\p u|^2\al)\\
\quad\quad\quad\quad\quad\quad\quad\quad+u\times (\De \al+|\p u|^2\al)+\bar{R},\\[1ex]
\frac{\p \al}{\p \nu}|_{\p\Om}=0,\\[1ex]
\al_0=\p_t \om_\ep(0): \Om\to \Real^K,
\end{cases}
\end{equation}
where we denote
\begin{equation}\label{formula-bar-R}
\begin{aligned}
	\bar{R}=&(\ep I+u\times)\p_tR_k+v_1\times R_k\\
	&-\<\om_\ep,v_2\>u-\<\om_\ep,v_1\>v_1\\
	&+v_1\times(\De \om_\ep+|\p u|^2\om_\ep)+2u\times \<\p u, \p v_1\>\om_\ep\\
	&+\ep(2\<\p v_1, \p \om_\ep\>u+\<\De v_1, \om_\ep\>u+2\<\p v_1,\p u\> \om_\ep)\\
	&+\ep(2\<\p u, \p \om_\ep\>+\<\De u, \om_\ep\>)v_1.
\end{aligned}
\end{equation}

Using the bounds \eqref{es-R1},\eqref{es-T-k} and \eqref{es-T-k1} with $k\geq 1$, along with the regularity property
\[\om_\ep\in  C^0([0,T_0], W^{3,2}(\Om))\cap L^2([0, T_0],W^{4,2}(\Om)),\]
then we derive from Lemma \ref{alg} that the remainder term $\bar{R}$ satisfies
\[\bar{R}\in C^0([0,T_0], L^2(\Om))\cap L^\infty([0,T_0],W^{1,2}(\Om)).\]

On the other hand, from the construction of equation \eqref{eq-para-v_k}, we derive the initial condition
\[\p_t \om_\ep(0)+\<v_k(0), v_1(0)\>u_0=-\ep u_0\times v_{k+1}(0)+v_{k+1}(0).\]
The compatibility conditions \eqref{comp-v-k} further imply that
\[\frac{\p \al_0}{\p \nu}|_{\p\Om}=\frac{\p}{\p \nu}\p_t \om_\ep(0)|_{\p\Om}=0.\]

Now, let
$$f_1=2\p u\otimes u,\quad f_2=\De u\otimes u+|\p u|^2I\quad \mbox{and}\quad f_3=\bar{R}.$$
Applying Theorem \ref{thm-h}, we get a solution $\al$ to the problem on $[0,T_0]$ satisfying
\[\p^i_t\al\in C^0([0,T_0], W^{2-2i,2}(\Om))\cap L^2([0,T_0], W^{3-2i,2}(\Om))\]
for $i=0,1$.

Finally, an argument of uniqueness yields
\[\p^i_t\om_\ep=\al\in  C^0([0,T_0], W^{2-2i,2}(\Om))\cap L^2([0,T_0], W^{3-2i,2}(\Om))\]
for $i=0,1$.
\medskip

\noindent\emph{Step 3: $W^{4,2}$-regularity of $\om_\ep$. }\

From equation \eqref{para-Nuemann-eq-v_1}, we deduce the following elliptic equation for $\om_\ep$:
\begin{align*}
\ep\De \om_\ep+u\times \De \om_\ep=&\p_t \om_\ep+\<\om_\ep, v_1\>u-u\times\om_\ep |\p u|^2-(\ep I+u\times)R_k\\
&-\ep(2\<\p u, \p \om_\ep\>u+\<\De u, \om_\ep\>u+|\p u|^2\om_\ep).
\end{align*}
Then applying the bounds \eqref{es-R1} and \eqref{es-T-k} with $k\geq 1$, along  with the regularity properties
\[\om_\ep\in  C^0([0,T_0], W^{3,2}(\Om))\cap L^2([0, T_0],W^{4,2}(\Om))\]
and
\[\p_t\om_\ep\in  C^0([0,T_0], W^{2,2}(\Om))\cap L^2([0,T_0], W^{3,2}(\Om)),\]
we establish that
\[\ep\De \om_\ep+u\times \De \om_\ep\in L^\infty([0,T_0], W^{2,2}(\Om))\cap L^2([0, T_0],W^{3,2}(\Om)).\]
This further implies that
\[\De \om_\ep \in L^\infty([0,T_0], W^{2,2}(\Om))\cap L^2([0,T_0], W^{3,2}(\Om)).\]
Then, applying Lemma \ref{eq-norm} yields
\[ \om_\ep \in L^\infty([0,T_0], W^{4,2}(\Om))\cap L^2([0,T_0], W^{5,2}(\Om)).\]

By Lemma \ref{C^0-em}, we finally conclude that
\[\om_\ep \in C^0([0,T_0], W^{4,2}(\Om))\cap L^2([0,T_0], W^{5,2}(\Om)).\]
\end{proof}

Furthermore, we can show that $\om_\ep(x,t)\in T_{u(x,t)}\U^2$, i.e., $\<\om_\ep, u\>(x,t)=0$, for all $(x,t)\in \bar{\Om}\times [0,T_0]$.

\begin{prop}\label{tangent-om-ep}
Let $\om_\ep\in C^{0}([0,T_0],W^{2,2}(\Om))\cap L^2([0,T_0], W^{3,2}(\Om))$ be a solution to the problem \eqref{para-Nuemann-eq-v_1}. Then for any $(x,t)\in \bar{\Om}\times [0,T_0]$, $\<\om_\ep, u\>(x,t)=0$.
\end{prop}
\begin{proof}
For simplicity, we decompose $\om_\ep$ as
\[\om_\ep=\om_\ep^{\top}+\om_\ep^{\perp},\]
where $\om_\ep^{\perp}$ denotes the vertical part of $\om_\ep$, i.e. $\om_\ep^{\perp}=\<\om_\ep,u\>u$, and $\om_\ep^{\top}$ is the tangent part.
	
By employing equation \eqref{para-Nuemann-eq-v_1}, a straightforward calculation yields
\begin{align*}
\frac{1}{2}\p_t\int_{\Om}|\om^{\perp}_\ep|^2dx=&\int_{\Om}\<\om_\ep, u\>(\<\p_t \om_\ep, u\>+\<\om_\ep, v_1\>)dx\\
=&\int_{\Om}\<\p_t \om_\ep+\<\om_\ep, v_1\>u, \om^{\perp}_\ep\>dx\\
=&\ep\int_{\Om}\<\De \om_\ep+2\<\p u,\p \om_\ep\>u+\<\De u, \om_\ep\>u+|\p u|^2 \om_\ep, \om^{\perp}_\ep\>dx\\
=&\ep\int_{\Om}\<\De \om^{\perp}_\ep+2\<\p u,\p \om^{\perp}_\ep\>u+\<\De u, \om^{\perp}_\ep\>u+|\p u|^2 \om_\ep^{\perp}, \om^{\perp}_\ep\>dx\\
\leq &-\ep\int_{\Om}|\p\om^{\perp}_\ep|^2dx+C\norm{u}_{W^{3,2}}\norm{\om^{\perp}_\ep}_{W^{1,2}}\norm{\om^{\perp}_\ep}_{L^2}+C\norm{u}^{2}_{W^{3,2}}\norm{\om^{\perp}_\ep}^2_{L^2}\\
\leq &C_\ep(\norm{u}^4_{W^{3,2}}+1)\norm{\om^{\perp}_\ep}^2_{L^2}.
\end{align*}
Here we have used the fact $R_k$ is a section of $u^*T\U^2$, and applied the following formula
\begin{align*}
&\De \om^{\top}_\ep+2\<\p u,\p \om^{\top}_\ep\>u+\<\De u, \om^{\top}_\ep\>u+|\p u|^2 \om^{\top}_\ep\\
=&-\n^*\n \om_\ep^{\top}+R^{\mathbb{S}^2}(\om^{\top}_\ep, \n_iu)\n_iu\in T_u\mathbb{S}^2
\end{align*}
to show
\begin{align*}
\int_{\Om}\<\De \om^{\top}_\ep+2\<\p u,\p \om^{\top}_\ep\>u+\<\De u, \om^{\top}_\ep\>u+|\p u|^2 \om_\ep^{\top}, \om^{\perp}_\ep\>dx=0.
\end{align*}
Since $\om_\ep^\perp(0)=0$, Gronwall's inequality implies that
\begin{align*}
\int_{\Om}|\om^{\perp}_\ep|^2dx(t)\leq C(T_0,\norm{u}_{L^\infty([0,T_0],W^{3,2})})\int_{\Om}|\om^{\perp}_\ep|^2dx(0)=0
\end{align*}
for any $0\leq t\leq T_0$.
\end{proof}

Proposition \ref{tangent-om-ep} tells us that $\om_\ep$ is a section of the bundle $u^*(T\U^2)$. Since equation \eqref{ex-eq-para-vk} is the extrinsic formulation of equation \eqref{eq-para-v_k}, then $\om_\ep$ also satisfies the intrinsic problem \eqref{para-Nuemann-eq-v_k}, i.e.
\begin{equation*}
\begin{cases}
\n_t \om_\ep=(\ep I+u\times )(-\n^*\n \om_\ep+R^{\U^2}(\om_\ep, \n_j u)\n_j u+R_k),\\[1ex]
\frac{\p \om_\ep}{\p \nu}|_{\p\Om}=0,\\[1ex]
\om_\ep(0)=v_k(0): \Om\to u^*_0(T\U^2).
\end{cases}
\end{equation*}
\medskip

\subsection{Uniform energy estimates for $\om_\ep$}\label{ss: uniform-en-es-1}\

In this part, we establish uniform $W^{4,2}$-energy estimates for $\om_\ep$ with respect to $\ep \in (0,1)$. We begin by proving several critical equivalent estimates for Sobolev norms of $\om_\ep$.

\begin{lem}\label{equiv-es-om-ep}
There exists a constant $C$ independent of $\ep$ such that the solution $\om_\ep$ of the problem \eqref{para-Nuemann-eq-v_1} obtained in Theorem \ref{W^{4,2}-solu-om-ep} satisfies
\begin{align}
\norm{\om_\ep}^2_{W^{2,2}}\leq &C(\norm{\p_t \om_\ep}^2_{L^2}+\norm{\om_\ep}_{L^2}+1),\label{equiv-es-om-1}\\
\norm{\om_\ep}^2_{W^{3,2}}\leq &C(\norm{\p_t \om_\ep}^2_{W^{1,2}}+\norm{\om_\ep}_{L^2}+1),\label{equiv-es-om-2}\\
\norm{\om_\ep}^2_{W^{4,2}}+\norm{\p_t \om_\ep}^2_{W^{2,2}}\leq &C(\norm{\p^2_t \om_\ep}^2_{L^2}+\norm{\p_t \om_\ep}^2_{L^2}+\norm{\om_\ep}_{L^2}+1).\label{equiv-es-om-3}
\end{align}
\end{lem}
\begin{proof}
Since $\om_\ep\in T_u\mathbb{S}^2$ is a solution to \eqref{para-Nuemann-eq-v_k} in $[0,T_0]$, satisfying estimates \eqref{W^{3,2}-es-om}, we get the following evolution equation.
\[\frac{1}{1+\ep^2}(\ep \n_t \om_\ep-u\times \n_t \om_\ep)=-\n^*\n \om_\ep+R^{\U^2}(\om_\ep, \n_j u)\n_j u+R_k.\]
Using formula \eqref{formula-laplace-and-curvature}, we further obtain
\begin{equation}\label{eq-De-om-ep}
\begin{aligned}
	\De \om_\ep=&\frac{1}{1+\ep^2}(\ep \p_t \om_\ep-u\times \p_t\om_\ep)+\ep\om_\ep\#v_1\#u\\
	&+\p u\#\p \om_\ep\# u+\De u\# \om_\ep\# u+\p u\#\p u\# \om_\ep+R_k.
\end{aligned}
\end{equation}

Then, applying the estimates \eqref{es-T-k} and \eqref{es-R1}, we derive from the formula \eqref{eq-De-om-ep} that
\begin{align*}
	\norm{\om_\ep}^2_{W^{2,2}}:=&\norm{\om_\ep}^2_{L^2}+\norm{\De \om_\ep}^2_{L^2}\\
	\leq &C(\norm{\p_t \om_\ep}^2_{L^2}+\norm{\om_\ep}^2_{W^{1,2}}+1)\\
	\leq &C(\norm{\p_t \om_\ep}^2_{L^2}+\norm{\om_\ep}^2_{L^2}+1)+\frac{1}{2}\norm{\om_\ep}^2_{W^{2,2}},
\end{align*}
where we used the following interpolation inequality
\[\norm{\p \om_\ep}_{L^2}\leq \norm{\om_\ep}^{1/2}_{W^{2,2}}\norm{\om_\ep}^{1/2}_{L^2}.\]
This implies the estimate \eqref{equiv-es-om-1}.

Taking the derivative for equation \eqref{eq-De-om-ep} gives
\begin{align*}
	\p \De \om_\ep=&\frac{1}{1+\ep^2}(\ep \p\p_t \om_\ep-u\times \p\p_t\om_\ep)+\p u\# \p_t \om_\ep+\ep\p(\om_\ep\#v_1\#u)\\
	&+\p(\p u\#\p \om_\ep\# u+\De u\# \om_\ep\# u+\p u\#\p u\# \om_\ep)+\p R_k.
\end{align*}
By estimates \eqref{es-T-k} and \eqref{es-R1}, it is not difficult to show
\begin{align*}
	\norm{\p \De \om_\ep}^2_{L^2}\leq &C(\norm{\p_t \om_\ep}^2_{W^{1,2}}+\norm{\om_\ep}^2_{W^{2,2}}+1)\\
	\leq &C(\norm{\p_t \om_\ep}^2_{W^{1,2}}+\norm{\om_\ep}^2_{L^2}+1),
\end{align*}
which yields the desired estimates \eqref{equiv-es-om-2}.

Finally, we establish the inequality \eqref{equiv-es-om-3}. Differentiating equation \eqref{eq-De-om-ep} twice yields the following identity:
\begin{align*}
	\p^2 \De \om_\ep=&\frac{1}{1+\ep^2}(\ep \p^2\p_t \om_\ep-u\times \p^2\p_t\om_\ep)+\p^2 u\# \p_t \om_\ep+\p u\# \p \p_t \om_\ep\\
	&+\ep\p^2(\om_\ep\#v_1\#u)+\p^2(\p u\#\p \om_\ep\# u+\De u\# \om_\ep\# u+\p u\#\p u\# \om_\ep)\\
	&+\p^2R_k.
\end{align*}
Then we can derive from the estimates \eqref{es-T-k} and \eqref{es-R1} that
\begin{equation}\label{es-p^2De-om-1}
	\begin{aligned}
		\norm{\p^2\De \om_\ep}^2_{L^2}\leq &C(\norm{\p_t \om_\ep}^2_{W^{2,2}}+\norm{\om_\ep}^2_{W^{3,2}}+1)\\
		\leq &C(\norm{\p_t \om_\ep}^2_{W^{2,2}}+\norm{\om_\ep}^2_{L^2}+1).
	\end{aligned}
\end{equation}

On the other hand, by \eqref{eq-De-om-ep}, $\p_t\om_\ep$ satisfies
\begin{equation}\label{eq-Dep_t-om}
\begin{aligned}
	\De \p_t\om_\ep=&\frac{1}{1+\ep^2}(\ep \p^2_t \om_\ep-u\times \p^2_t\om_\ep)+v_1\# \p_t \om_\ep+\ep\p_t(\om_\ep\#v_1\#u)\\
	&+\p_t(\p u\#\p \om_\ep\# u+\De u\# \om_\ep\# u+\p u\#\p u\# \om_\ep)+\p_tR_k,
\end{aligned}
\end{equation}
then we apply the H\"older inequality to show
\begin{equation}\label{es-p^2De-om-2}
	\begin{aligned}
	\norm{\De \p_t \om_\ep}^2_{L^2}\leq &C\norm{\p^2_t \om_\ep}^2_{L^2}+C\norm{\p_t \om_\ep}^2_{W^{1,2}}+C\norm{\p_t R_k}^2_{L^2}\\
	&+C(\norm{v_2}^2_{L^2}+\norm{\De v_1}^2_{L^2})\norm{\om_\ep}^2_{L^\infty}+C\norm{\om_\ep}^2_{W^{3,2}}\\
	\leq&C\norm{\p^2_t \om_\ep}^2_{L^2}+C\norm{\p_t \om_\ep}^2_{W^{1,2}} +C\norm{\om_\ep}^2_{W^{3,2}}+C\\
	\leq &C\norm{\p^2_t \om_\ep}^2_{L^2}+C\norm{\p_t \om_\ep}^2_{L^2} +C\norm{\om_\ep}^2_{L^2}+\frac{1}{2}\norm{\De \p_t \om_\ep}^2_{L^2}+C,
	\end{aligned}
\end{equation}
Here we applied the following interpolation inequality:
\[\norm{\p \p_t\om_\ep}_{L^2}\leq \norm{\p_t\om_\ep}^{1/2}_{W^{2,2}}\norm{\p_t\om_\ep}^{1/2}_{L^2}\leq C\(\norm{\p_t\om_\ep}_{L^2}+\norm{\De \p_t \om_\ep}_{L^2}\)^{1/2}\norm{\p_t\om_\ep}^{1/2}_{L^2}.\]

Combining estimates \eqref{es-p^2De-om-1} with \eqref{es-p^2De-om-2}, we employ Lemma \eqref{eq-norm} to obtain the desired bound \eqref{equiv-es-om-3}.
\end{proof}

We also need to show a geometric property concerning the special structure of $\U^2$.
\begin{lem}\label{alg-sphere}
For any $p\in \mathbb{S}^2$ and $X_i\in T_{p}\mathbb{S}^2$ with $i=1,2,3$, we have
	\[\<X_1\times X_2, X_3\>=0.\]
\end{lem}
\begin{proof}
If there exists a  constant $\lambda$ such that $X_1=\lambda X_2$, then the cross product vanishes:
\[X_1\times X_2=0.\]
So, we may assume that $X_1$ and $X_2$ are linearly independent. Then there exists constant $\lambda_i$ with $i=1,2$, such that
	\[X_3=\lambda_1X_1+\lambda_2X_2,\]
which implies that
\[\<X_1\times X_2, X_3\>=0.\]
\end{proof}

\subsubsection{\bf{Uniform $L^2$-estimates of $\om_\ep$.}} Now, we establish uniform energy estimates for $\om_\ep$. Since $\frac{\p \om_\ep}{\p \nu}|_{\p\Om\times[0,T_0]}=0$ and $\om_\ep\in T_u\U^2$, testing equation \eqref{para-Nuemann-eq-v_1} with $\om_\ep$ gives
\begin{equation}\label{energy-es-om-1}
\begin{aligned}
\frac{1}{2}\p_t \int_{\Om}|\om_\ep|^2dx\leq &-\ep\int_{\Om}|\p \om_\ep|^2dx-\int_{\Om}\<\p_i u\times \p_i \om_\ep, \om_\ep\>dx+C\int_{\Om}|R_k||\om_\ep|dx\\
&+C\ep\int_{\Om}|\De u| |\om_\ep|^2dx+C\ep\int_{\Om}|\p u| |\p \om_\ep||\om_\ep|dx+C\int_{\Om}|\p u|^2|\om_\ep|^2dx\\
\leq &-\frac{\ep}{2}\int_{\Om}|\p \om_\ep|^2dx+C(\norm{u}^2_{W^{4,2}}+1)(\norm{\om_\ep}^2_{L^2}+1),
\end{aligned}
\end{equation}
where we have applied Lemma \ref{alg-sphere} to show
\begin{align*}
\<\p_i u\times \p_i \om_\ep, \om_\ep\>=&\<\p_i u\times (\p_i \om_\ep)^\perp,\om_\ep\>\\
=&\<\p_i u\times u,\om_\ep\>\<\p_i\om_\ep, u\>\\
=&-\<\p_i u\times u,\om_\ep\>\<\om_\ep, \p_iu\>.
\end{align*}

Using estimate \eqref{es-T-k}, we derive from Gronwall's inequality that
\begin{align}
\sup_{0\leq t\leq T_0}\norm{\om_\ep}^2_{L^2}\leq C(T_0)(\norm{v_k(0)}^2_{L^2}+1)\leq C(T_0, \norm{u_0}^2_{W^{2k,2}}).\label{L^2-bound-om}
\end{align}

\subsubsection{\bf{Uniform $W^{2,2}$-estimates of $\om_\ep$.}} Next, we establish uniform $W^{2,2}$-energy estimates for $\om_\ep$. To this end, we consider the equation satisfied by $\p_t \om_\ep$:
\begin{equation}\label{eq-p_t-om}
\begin{aligned}
	\p_t \p_t \om_\ep+\<\p_t \om_\ep, v_1\>u=&\ep(\De\p_t \om_\ep+2\<\p u, \p \p_t \om_\ep\>u+\<\De u, \p_t \om_\ep\>u+|\p u|^2\p_t \om_\ep)\\
	&+u\times (\De \p_t \om_\ep+|\p u|^2\p_t \om_\ep)+\bar{R},
\end{aligned}
\end{equation}
where the remainder term $\bar{R}$ is defined in \eqref{formula-bar-R}.

The boundary condition $\frac{\p \om_\ep}{\p \nu}|_{\p\Om\times[0,T_0]}=0$ allows us to apply Lemma \ref{comp-cond}, which yields the corresponding Neumann condition for the time derivative:
\[\frac{\p }{\p \nu}\p_t\om_\ep|_{\p \Om\times[0,T_0]}=0.\]

Then, taking $\p_t \om_\ep$ as a test function for equation \eqref{eq-p_t-om}, we obtain
\begin{equation}\label{energy-es-om-2}
\begin{aligned}
\frac{1}{2}\p_t \int_{\Om}|\p_t\om_\ep|^2dx\leq &-\ep\int_{\Om}|\p\p_t \om_\ep|^2dx+C\int_{\Om}|v_1||\p_t\om_\ep|^2dx\\
&+\int_{\Om}\<u\times \De \p_t\om_\ep, \p_t \om_\ep\>dx\\
&+C\ep\int_{\Om}|\De u| |\p_t\om_\ep|^2dx+C\ep\int_{\Om}|\p u| |\p\p_t \om_\ep||\p_t\om_\ep|dx\\
&+C\ep\int_{\Om}|\p u|^2|\p_t\om_\ep|^2dx+C\int_{\Om}|\bar{R}||\p_t \om_\ep|dx\\
\leq &-\frac{\ep}{2}\int_{\Om}|\p\p_t \om_\ep|^2dx+\int_{\Om}\<u\times \De \p_t\om_\ep, \p_t \om_\ep\>dx\\
&+C(\norm{u}^2_{W^{3,2}}+\norm{v_1}^2_{W^{2,2}}+1)\norm{\p_t\om_\ep}^2_{L^2}+C\norm{\bar{R}}^2_{L^2}.
\end{aligned}
\end{equation}
Here, the third term on the right hand side of \eqref{energy-es-om-2} admits the following estimate:
\begin{align*}
&\int_{\Om}\<u\times \De \p_t\om_\ep, \p_t \om_\ep\>dx=\int_{\Om}\<u\times \De \p_t\om_\ep, (\p_t \om_\ep)^{\top}\>dx\\
=&-\int_{\Om}\<u\times \p \p_t\om_\ep, \p((\p_t \om_\ep)^{\top})\>dx-\int_{\Om}\<\p u\times (\p \p_t\om_\ep)^\perp, (\p_t \om_\ep)^{\top}\>dx\\
=&\int_{\Om}\<u\times \p \p_t\om_\ep, \p(\<\p_t \om_\ep, u\>u)\>dx+\int_{\Om}\<\p u\times u, (\p_t \om_\ep)^{\top}\>\<\p u, \p_t\om_\ep\>dx\\
&+\int_{\Om}\<\p u\times u, (\p_t \om_\ep)^{\top}\>\p\<v_1, \om_\ep\>dx\\
=&-\int_{\Om}\<u\times \p \p_t\om_\ep, \p(\<\om_\ep, v_1\>u)\>dx+\int_{\Om}\<\p u\times u, \p_t \om_\ep\>\<\p u, \p_t\om_\ep\>dx\\
&+\int_{\Om}\<\p u\times u, \p_t \om_\ep\>\p\<v_1, \om_\ep\>dx\\
=&-\int_{\Om}\<u\times \p \p_t\om_\ep, \<\om_\ep, v_1\>\p u\>dx+\int_{\Om}\<\p u\times u, \p_t \om_\ep\>\<\p u, \p_t\om_\ep\>dx\\
&+\int_{\Om}\<\p u\times u, \p_t \om_\ep\>\p\<v_1, \om_\ep\>dx\\
=&-\int_{\Om}\<\p_t\om_\ep, \p_i(\<\om_\ep, v_1\>u\times \p_i u)\>dx+\int_{\Om}\<\p u\times u, \p_t \om_\ep\>\<\p u, \p_t\om_\ep\>dx\\
&+\int_{\Om}\<\p u\times u, \p_t \om_\ep\>\p\<v_1, \om_\ep\>dx\\
\leq &\int_{\Om}(|\p \om_\ep||v_1||\p u|+|\om_\ep||v_1|^2+|\om_\ep||\p v_1||\p u|)|\p_t \om_\ep|dx+C\int_{\Om}|\p u|^2|\p_t \om_\ep|^2dx\\
\leq &C(\norm{u}^2_{W^{3,2}}+\norm{v_1}^2_{W^{2,2}}+1)^2\norm{\om_\ep}^2_{W^{1,2}}+C\norm{\p_t \om_\ep}^2_{L^2}\\
\leq &C(\norm{\p_t \om_\ep}^2_{L^2}+\norm{\om_\ep}^2_{L^2}),
\end{align*}
where we employed Lemma \ref{alg-sphere} to show
\[\<\p u\times (\p \p_t\om_\ep)^{\top}, (\p_t \om)^{\top}\>=0.\]

It remains to control the $L^2$-norm of the term $\bar{R}$. From its explicit expression given in \eqref{formula-bar-R}, we decompose $\bar{R}$ as
\begin{align*}
\bar{R}=&(\ep I+u\times)\p_tR_k+v_1\times R_k\\
&-\<\om_\ep,v_2\>u-\<\om_\ep,v_1\>v_1\\
&+v_1\times(\De \om_\ep+|\p u|^2\om_\ep)+2u\times \<\p u, \p v_1\>\om_\ep\\
&+\ep(2\<\p v_1, \p \om_\ep\>u+\<\De v_1, \om_\ep\>u+2\<\p v_1,\p u\> \om_\ep)\\
&+\ep(2\<\p u, \p \om_\ep\>+\<\De u, \om_\ep\>)v_1.
\end{align*}
Then we derive from estimates \eqref{es-T-k} and \eqref{es-R1} that
\begin{align*}
\norm{\bar{R}}^2_{L^2}\leq& C(\norm{u}^2_{W^{3,2}}+\norm{v_1}^2_{W^{2,2}}+\norm{v_2}^2_{L^2}+1)^2\norm{\om_\ep}^2_{W^{2,2}}\\
&+C(\norm{R_k}^2_{L^2}+\norm{\p_t R_k}^2_{L^2})\\
\leq &C(\norm{\p_t \om_\ep}^2_{L^2}+\norm{\om_\ep}^2_{L^2}+1).
\end{align*}

Therefore, substituting these estimates into \eqref{energy-es-om-2}, we obtain
\begin{align*}
\frac{1}{2}\p_t \int_{\Om}|\p_t\om_\ep|^2dx\leq C(\norm{\p_t \om_\ep}^2_{L^2}+\norm{\om_\ep}^2_{L^2}+1).
\end{align*}
Now, by applying the equivalent estimates \eqref{equiv-es-om-1} and Gronwall's inequality, we obtain the following uniform bound:
\begin{align}
\sup_{0\leq t\leq T_0}(\norm{\om_\ep}^2_{W^{2,2}}+\norm{\p_t \om_\ep}^2_{L^2})\leq C(T_0, \norm{u_0}^2_{W^{2k+2}}).\label{W^{2,2}-bound-om}
\end{align}
Here we used the following expression of the initial data $\p_t\om_\ep(0)$
\begin{align*}
\p_t \om_\ep(0)=&-\<v_k(0), v_1(0)\>u_0\\
&+\ep(\De v_k(0)+2\<\p u_0, \p v_k(0)\>u+\<\De u_0, v_k(0)\>u+|\p u_0|^2v_k(0))\\
&+u_0\times (\De v_k(0)+|\p u_0|^2v_k(0))+(\ep I+u_0\times)R_k(0)
\end{align*}
to establish the estimate:
\[\norm{\p_t \om_\ep(0)}^2_{L^2(\Om)}\leq C(\norm{u_0}_{W^{2k+2,2}}).\]

\subsubsection{\bf{Uniform $W^{4,2}$-estimates of $\om_\ep$.}} By estimates \eqref{equiv-es-om-3} and \eqref{W^{2,2}-bound-om}, to establish uniform $W^{4,2}$-estimates of $\om_\ep$, it suffices to obtain a uniform bound for the $L^2$-norm of $\p^2_t \om_\ep$. For this purpose, we analyze the equation governing $\p^2_t \om_\ep$.

Since $\om_\ep\in T_u\mathbb{S}^2$ satisfies $\<v_1,\om_\ep\>=-\<u, \p_t \om_\ep\>$, we can derive from equation \eqref{para-Nuemann-eq-v_1} that
\begin{align*}
\p_t \om_\ep-\<u, \p_t \om_\ep\>u=&\ep(\De \om_\ep+\p_i(\<\p_i u, \om_\ep\>u)+\p u\# \p \om_\ep \#u+|\p u|^2\om_\ep)\\
&+u\times (\De \om_\ep+|\p u|^2\om_\ep)+(\ep I+ u\times) R_k.
\end{align*}
Differentiating this equation twice with respect to $t$ gives
\begin{align*}
 \p^2_t\p_t \om_\ep-\p^2_t(\<u, \p_t \om_\ep\>u)=&\ep\De \p^2_t\om_\ep+u\times \De \p_t^2\om_\ep+\ep\p_i\p^2_t(\<\p_i u, \om_\ep\>u)\\
 &+\ep\p^2_t(\p u\#\p \om_\ep\# u+\p u\#\p u\# \om_\ep)\\
 &+v_2\#\De \om_\ep+v_1\# \De \p_t\om_\ep+\p^2_t(|\p u|^2u\times \om_\ep)\\
 &+(\ep I+u\times)\p^2_tR_k+v_2\# R_k+v_1\# \p_t R_k,
\end{align*}
where the second term on the left hand side expands as
\begin{align*}
\p^2_t(\<u, \p_t \om_\ep\>u)=&\p_t(\<\p^2_t \om_\ep, u\>u)+v_1\#\p^2_t\om_\ep\# u\\
&+\<v_2,\p_t\om_\ep\>u+v_1\#v_2\#\om_\ep+v_1\# v_1\#\p_t\om_\ep.
\end{align*}

Let $V=\p^2_t \om_\ep$. The regularity estimates \eqref{W^{4,2}-es-om-ep} imply that
\[V\in C^0([0,T_0],L^2(\Om))\cap L^2([0,T_0],W^{1,2}(\Om)),\]
hence $V$ is a weak solution to the Neumann problem:
\begin{equation}\label{para-Nuemann-eq-p-2-t-om}
\begin{cases}
\p_t V^\top=\ep\De V+u\times \De V+V\# v_1\# u+\ep\mbox{div}(\<\p u, V\>u)\\
\quad\quad\quad\,\,+\ep(\p u\# u\# \p V+\p u\#\p u\#V)+\ep \mbox{div}F+\bar{R}_1,\\[1ex]
\frac{\p V}{\p \nu}|_{\p\Om}=0,\\[1ex]
V(0)=\p^2_t \om_\ep(0): \Om\to \Real^K.
\end{cases}
\end{equation}
Namely, for any $\psi\in W^{1,2}_{1}([0,T_0]\times \Om, \Real^K)$, the following integral equality holds:
\begin{align*}
&\int_{\Om}\<V^\top, \psi\>dx(T)-\int_{\Om}\<V^\top, \psi\>dx(0)-\int_{0}^{T}\int_{\Om}\<V^\top, \p_t \psi\>dxdt\\
=&-\ep\int_{0}^{T}\int_{\Om}\<\p V, \p \psi\>dxdt-\int_{0}^{T}\int_{\Om}\<u\times \p V, \p \psi\>dxdt\\
&-\int_{0}^{T}\int_{\Om}\<\p u\times \p V, \psi\>dxdt+\int_0^T\int_{\Om}\<V\# v_1\# u, \psi\>dxdt\\
&-\ep\int_0^T\int_{\Om}\<\<V,\p u\>u,\p \psi \>dxdt+\ep\int_{0}^{T}\int_{\Om}\<\p u\# u\# \p V+\p u\#\p u\#V, \psi\>dxdt\\
&-\ep\int_0^T\int_{\Om}\<F,\p \psi \>dxdt+\int_0^T\int_{\Om}\<\bar{R}_1,\psi \>dxdt.
\end{align*}
Here we denote the tangent part of $V$ by
\[V^\top=V-\<V,u\>u,\]
and the terms $F$ and $\bar{R}_1$ are given by
\begin{align*}
F=&\p^2_t(\<\p u, \om_\ep\>u)-\<\p u, V\>u\\
=&\p v_2\# \om_\ep\# u+\p u\# \om_\ep \# v_2\\
&+\p v_1\# \p_t \om_\ep\# u+\p u\# \p_t \om_\ep\# v_1+\p v_1\# \om_\ep\# v_1,\\
\bar{R}_1=&\<v_2,\p_t\om_\ep\>u+v_1\#v_2\#\om_\ep+v_1\# v_1\#\p_t\om_\ep\\
&+\ep(\p v_2\# \p \om_\ep\# u+\p u\# \p\om_\ep\# v_2)\\
&+\ep(\p v_1\# \p\p_t \om_\ep\# u+\p v_1\# \p\om_\ep\# u+\p u\# \p_t \om_\ep\# v_1)\\
&+\ep(\p v_2\# \p u\# \om_\ep+\p u\# \p v_1\# \p_t \om_\ep+\p v_1\#\p v_1\# \om_\ep)\\
&+v_2\#\De \om_\ep+v_1\# \De \p_t\om_\ep+\p^2_t(|\p u|^2u\times \om_\ep)\\
&(\ep I+u\times)\p^2_tR_k+v_2\# R_k+v_1\# \p_t R_k.
\end{align*}
Additionally, by estimates \eqref{es-T-k}, \eqref{es-R} and \eqref{W^{4,2}-es-om-ep}, we deduce 
\[F\in L^\infty([0,T_0], L^2(\Om))\quad\mbox{and}\quad \bar{R}_1\in L^\infty([0,T_0], L^2(\Om)).\]

Then following analogous arguments as in \cite{YZ} (see the proof of Theorem 2.2 on page 48 of \cite{YZ}), we take $\psi=V^{\top}_h$, i.e., the Steklov average of $V^\top$ defined by
\[V^{\top}_h(x,t)=\frac{1}{2h}\int_{t-h}^{t+h}V^{\top}(x,s)ds\]
for sufficiently small $h>0$, as a test function for equation \eqref{para-Nuemann-eq-p-2-t-om}. Passing to the limit as $h\to 0$ yields
\begin{equation}\label{formula-V}
\begin{aligned}
&\frac{1}{2}\(\int_{\Om}|V^\top|^2dx(T)-\int_{\Om}|V^\top|^2 dx(0)\)\\
=&-\ep\int_{0}^{T}\int_{\Om}\<\p V, \p V^{\top}\>dxdt-\int_{0}^{T}\int_{\Om}\<u\times \p V, \p V^{\top}\>dxdt\\
&-\int_{0}^{T}\int_{\Om}\<\p u\times \p V, V^{\top}\>dxdt+\int_0^T\int_{\Om}\<V\# v_1\# u, V^{\top}\>dxdt\\
&-\ep\int_0^T\int_{\Om}\<\<V,\p u\>u,\p V^{\top} \>dxdt+\ep\int_{0}^{T}\int_{\Om}\<\p u\# u\# \p V+\p u\#\p u\#V, V^{\top}\>dxdt\\
&-\ep\int_0^T\int_{\Om}\<F,\p V^{\top} \>dxdt+\int_0^T\int_{\Om}\<\bar{R}_1,V^{\top} \>dxdt\\
=&M_1+\cdots+M_8.
\end{aligned}
\end{equation}

The terms $M_1$-$M_8$ admit the following estimates. We begin by analyzing the term $M_1$:
\begin{align*}
M_1=&-\ep\int_{0}^{T}\int_{\Om}|\p V^{\top}|^2dxdt-\ep\int_{0}^{T}\int_{\Om}\<\p V^{\perp}, \p V^{\top}\>dxdt\\
\leq &-\frac{3\ep}{4}\int_{0}^{T}\int_{\Om}|\p V^{\top}|^2dxdt+C\ep\int_{0}^{T}\int_{\Om}|\p V^{\perp}|^2dxdt\\
\leq &-\frac{3\ep}{4}\int_{0}^{T}\int_{\Om}|\p V^{\top}|^2dxdt+C\ep\int_{0}^{T}(\norm{V^\top}^2_{L^2}+1)dt.
\end{align*}
where we have used the decomposition
\begin{align*}
V^{\perp}=&\<V,u\> u\\
=&-\<\om_\ep, v_2\>u-2\<\p_t\om_\ep, v_1 \>u,
\end{align*}
along with the estimate  \eqref{equiv-es-om-1} to obtain the following bounds
\begin{equation}\label{es-V^perp}
\begin{aligned}
\int_{\Om}|V^\perp|^2dx\leq &C\norm{v_2}^2_{L^2}\norm{\om_\ep}^2_{W^{2,2}}\\
&+C\norm{\p_t \om_\ep}^2_{L^2}\norm{v_1}^2_{W^{2,2}}\leq C\\
\int_{\Om}|\p V^{\perp}|^2dx\leq &C\norm{v_2}^2_{W^{1,2}}\norm{\om_\ep}^2_{W^{2,2}}+C\norm{v_1}^{2}_{W^{2,2}}\norm{\p_t\om_\ep}^2_{W^{1,2}}\\
\leq &C(\norm{V}^2_{L^2}+\norm{\p_t \om_\ep}^2_{L^2}+\norm{\om_\ep}^2_{L^2}+1)\\
\leq &C(\norm{V^\top}^2_{L^2}+1).
\end{aligned}
\end{equation}

Applying estimates \eqref{es-T-k} and \eqref{equiv-es-om-1}, we analyze the term $M_2$:
\begin{align*}
	M_2=&-\int_{0}^{T}\int_{\Om}\<u\times \p V, \p (\<u,V\>u)\>dxdt\\
	=&-\int_{0}^{T}\int_{\Om}\<u\times \p V, \<u,V\>\p u\>dxdt\\
	=&\int_{0}^{T}\int_{\Om}\<u\times \p V, (\<v_2,\om_\ep\>+2\<v_1, \p_t \om_\ep\>)\p u\>dxdt\\
	=&\int_{0}^{T}\int_{\Om}\< V, \p_i\((\<v_2,\om_\ep\>+2\<v_1, \p_t \om_\ep\>) u\times \p_i u\>\)dxdt\\
	\leq &C\int_{0}^{T}\norm{V}^2_{L^2}dt+C\int_{0}^{T}(\norm{\om_\ep}^2_{W^{2,2}}+\norm{\p_t \om_\ep}^2_{W^{1,2}})dt\\
	\leq &C\int_{0}^{T}(\norm{V^{\top}}^2_{L^2}+1)dt.
\end{align*}

We then apply Lemma \ref{alg-sphere} to obtain a bound of $M_3$:
\begin{align*}
|M_3|=&|\int_{0}^{T}\int_{\Om}\<\p u\times (\p V)^{\perp}, V^{\top}\>dxdt|\\
\leq &C\int_{0}^{T}\norm{V^\top}^2_{L^2}dt+C\int_{0}^{T}(\norm{\p V^\perp}^2_{L^2}+\norm{V}^2_{L^2})dt\\
\leq &C\int_{0}^{T}(\norm{V^{\top}}^2_{L^2}+1)dt,
\end{align*}
where we used the following formula:
\[(\p V)^\perp=\p V^\perp+V\# \p u\# u\]
and estimates \eqref{es-V^perp}.

The rest terms $M_4$-$M_8$ are estimated in the following.
\begin{align*}
|M_4|\leq &C\int_0^T\norm{v_1}_{W^{2,2}}\norm{V}^2_{L^2}dt\leq C\int_{0}^{T}(\norm{V^{\top}}^2_{L^2}+1)dt
\end{align*}
\begin{align*}
|M_5+M_6|\leq &C\ep\int_{0}^{T}(\norm{V^\top}^2_{L^2}+1)dt+\frac{\ep}{4}\int_{0}^{T}\norm{\p V^{\top}}^2_{L^2}dt,
\end{align*}
\begin{align*}
|M_7|\leq &C\ep\int_{0}^{T}\norm{F}^2_{L^2}dt+\frac{\ep}{4}\int_{0}^{T}\norm{\p V^{\top}}^2_{L^2}dt\\
\leq &C\ep\int_{0}^{T}(\norm{\om_\ep}^2_{W^{2,2}}+\norm{\p_t \om_\ep}^2_{L^2})dt+\frac{\ep}{4}\int_{0}^{T}\norm{\p V^{\top}}^2_{L^2}dt,
\end{align*}
and
\begin{align*}
|M_8|\leq &C\int_{0}^{T}\norm{\bar{R}_1-\<v_2,\p_t\om_\ep\>u}^2_{L^2}dt+C\int_{0}^{T}\norm{V^\top}^2_{L^2}dt\\
\leq &C\int_{0}^{T}(\norm{\om_\ep}^2_{W^{4,2}}+\norm{\p_t \om_\ep}^2_{W^{2,2}}+\sum_{i=1}^2\norm{\p^i_t R_k}^2_{W^{4-2i,2}})dt+C\int_{0}^{T}\norm{V^\top}^2_{L^2}dt\\
\leq &C\int_{0}^{T}(\norm{V^{\top}}^2_{L^2}+1)dt,
\end{align*}
where we have used the estimates \eqref{es-T-k}, \eqref{es-R1},\eqref{W^{2,2}-bound-om} and \eqref{es-V^perp}.

Therefore, substituting the above estimates for $M_1$-$M_7$ into \eqref{formula-V}, we obtain
\begin{align*}
\(\int_{\Om}|V^\top|^2dx(T)-\int_{\Om}|V^\top|^2 dx(0)\)\leq C	\int_{0}^{T}\int_{\Om}|V^{\top}|^2dx+1)dt
\end{align*}
for all $0\leq T\leq T_0$. Then Gronwall's inequality implies that
\[\sup_{0\leq t\leq T_0}\norm{V^\top}^2_{L^2}\leq C(T_0, \norm{V(0)}^2_{L^2})\leq C(T_0, \norm{u_0}_{W^{2k+4,2}}).\]
By estimates \eqref{equiv-es-om-3} and \eqref{es-V^perp}, we deduce that
\begin{align*}
\sup_{0\leq t\leq T_0}\norm{\om_\ep}^2_{W^{4,2}}+\norm{\p_t \om_\ep}^2_{W^{2,2}}+\norm{\p_t^2\om_\ep}^2_{L^2}\leq C.
\end{align*}

We now summarize our main estimates for $\om_\ep$ in the following theorem
\begin{thm}\label{unif-W^{4,2}-es-om}
Assume that $u_0\in W^{2(k+1)+3,2}(\Om)$ satisfies the $(k+1)$-order compatibility conditions \eqref{com-cond-1}, and $u$ is a regular solution to \eqref{eq-LL} satisfying the property $\T_k$ (i.e., estimates \eqref{es-case-k}). Let $\om_\ep$ be given in Theorem \eqref{W^{4,2}-solu-om-ep}. Then there exists a constant $C$ depending only on $T_0$ and $\norm{u_0}_{W^{2k+4}}$ such that
\begin{align}
	\sup_{0\leq t\leq T_0}\norm{\om_\ep}^2_{W^{4,2}}+\norm{\p_t \om_\ep}^2_{W^{2,2}}+\norm{\p_t^2\om_\ep}^2_{L^2}\leq C.\label{W^{4,2}-bound-om}
\end{align}
\end{thm}

\medskip
\subsection{$W^{5,2}$-regularity of $\om_\ep$.}\

Recall that $\om_\ep$ is a solution to the parabolic approximation problem \eqref{para-Nuemann-eq-v_1} and satisfies the uniform bound \eqref{W^{4,2}-bound-om}. Applying the compactness lemma \eqref{A-S}, there exists a subsequence of $\{\om_\ep\}$ converging to a limiting map $\om$, which satisfies the energy bound
\[\sup_{0\leq t\leq T_0}\norm{\om}^2_{W^{4,2}}+\norm{\p_t \om}^2_{W^{2,2}}+\norm{\p_t^2\om}^2_{L^2}\leq C,\]
and solves the following Neumann problem:
\begin{equation}\label{Nuemann-eq-v_1}
	\begin{cases}
		\p_t \om+\<\om, v_1\>u=u\times (\De \om+|\p u|^2\om+R_k),\\[1ex]
		\frac{\p \om}{\p \nu}|_{\p\Om}=0,\\[1ex]
		\om_\ep(0)=v_k(0): \Om\to \Real^K.
	\end{cases}
\end{equation}

By property $\T_k$, we notice that $v_k\in L^\infty([0,T_0], W^{3,2}(\Om))$ with $\p_t v_k\in L^\infty([0,T_0], W^{1,2}(\Om))$ is another regular solution to \eqref{Nuemann-eq-v_1}. Defining $\beta=\om-v_k\in u^*T\U^2$, we apply Lemma \ref{alg-sphere} to show
\begin{align*}
\frac{1}{2}\p_t\int_{\Om}|\beta|^2dx=&\int_{\Om}\<u\times \De \beta, \beta \>dx\\
= &-\int_{\Om}\<\p u\times \p \beta, \beta\>dx\\
=&-\int_{\Om}\<\p u\times u \<u,\p \beta\>, \beta\>dx\\
=&\int_{\Om}\<\p u,\beta\>\<\p u\times u, \beta\>dx\\
\leq &C\norm{u}^2_{W^{3,2}}\int_{\Om}|\beta|^2dx.
\end{align*}
Since $\beta(0)\equiv 0$, Gronwall's inequality implies that $v_k=\om$. Consequently, for $i=0, 1, 2$, there holds true
\[\p^i_tv_k\in L^\infty([0,T_0], W^{4-2i,2}(\Om)).\]

By Lemma \ref{equiv-es-p-t-u}, this further implies that
\begin{align}
v_{k+i}\in L^\infty([0,T_0],W^{4-2i,2}(\Om)),\label{es-v-k-new}
\end{align}
for $0 \leq i\leq2$.

In particular, from the equation
\[\De u=-u\times v_1-|\p u|^2 u,\]
we derive from the special case $k=1$ in \eqref{es-v-k-new} that
\begin{align}
v_i=\n^i_t u\in L^\infty([0,T_0],W^{6-2i}(\Om))\label{es-u-new}
\end{align}
for $0\leq i\leq 3$.

These improved estimates \eqref{es-v-k-new} and \eqref{es-u-new} lead to higher regularity for $\om_\ep$. Precisely, we have the following result.

\begin{prop}\label{W^{5,2}-solu-om-ep}
Under the same assumptions as in Theorem \ref{W^{4,2}-solu-om-ep}, the solution $\om_\ep$ satisfies
\begin{align}
\p^i_t\om_\ep\in C^0([0,T_0], W^{5-2i}(\Om))\cap L^2([0,T_0], W^{6-2i}(\Om))\label{W^{5,2}-es-om-ep}	
\end{align}
for $0\leq i\leq 2$.
\end{prop}

\begin{proof}
Using estimates \eqref{es-T-k} and \eqref{es-u-new}, we can establish the following regularity properties for the remainder term $\bar{R}$ in equation \eqref{para-Nuemann-eq-p-t-om}:
\begin{align*}
\bar{R}&\in L^\infty([0,T_0], W^{2,2}(\Om)),\\
\p_t\bar{R}&\in L^\infty([0,T_0], L^2(\Om)).
\end{align*}

Now, applying Theorem \ref{thm-h} with the choices:
$$f_1=2\p u\otimes u,\quad f_2=\De u\otimes u+|\p u|^2I \quad\mbox{and}\quad f_3=\bar{R},$$
we obtain that $\p_t \om_\ep$ satisfies
\[\p^i_t\om_\ep\in C^0([0,T_0], W^{5-2i,2}(\Om))\cap L^2([0,T_0], W^{6-2i,2}(\Om))\]
for $i=1,2$. By employing the formula \eqref{eq-De-om-ep}, i.e.,
\begin{align*}
\De \om_\ep=&\frac{1}{1+\ep^2}(\ep \p_t \om_\ep-u\times \p_t\om_\ep)+\ep\om_\ep\#v_1\#u\\
&+\p u\#\p \om_\ep\# u+\De u\# \om_\ep\# u+\p u\#\p u\# \om_\ep+R_k,
\end{align*}
this estimate yields the enhanced spatial regularity of $\om_\ep$:
\[\om_\ep\in C^0([0,T_0],W^{5,2}(\Om))\cap L^2([0,T_0],W^{6,2}(\Om).\]
\end{proof}

A direct conclusion of Proposition \ref{W^{5,2}-solu-om-ep} is that
\[V=\p^2_t \om_\ep\in C^0([0,T_0],W^{1,2}(\Om))\cap L^2([0,T_0],W^{2,2}(\Om)),\]
hence $V$ is a strong solution to the Neumann problem \eqref{para-Nuemann-eq-p-2-t-om}:
\begin{equation}\label{para-Nuemann-eq-p-2-t-om'}
\begin{cases}
\p_t V^\top=\ep\De V+u\times \De V+V\# v_1\# u+\ep\mbox{div}(\<\p u, V\>u)\\
\quad\quad\quad\,\,+\ep(\p u\# u\# \p V+\p u\#\p u\#V)+\ep \mbox{div}F+\bar{R}_1,\\[1ex]
\frac{\p V}{\p \nu}|_{\p\Om}=0,\\[1ex]
V(0)=\p^2_t \om_\ep(0): \Om\to \Real^K.
\end{cases}
\end{equation}

\section{The proof of the main theorem \ref{main-thm}}\label{s: proof-main-thm}\

In this section, we derive uniform $W^{5,2}$-energy estimates for the family of solutions $\omega_\varepsilon$. Building upon these a priori estimates, we then complete the proof of property $\T_{k+1}$. Finally, we prove the main theorem \ref{main-thm}.

\subsection{Uniform $W^{5,2}$-energy estimates for $\om_\ep$}\label{ss: uniform-en-es-2}\

To establish uniform $W^{5,2}$-energy estimates for $\om_\ep$, we first derive  refined estimates for $R_k$ and then prove equivalent bounds for $W^{5,2}$-norms of $\om_\ep$.
\begin{lem}
Assume that the solution $u$ satisfies the estimates \eqref{es-T-k} and \eqref{es-v-k-new} with $k\geq 1$, then we have
\begin{align}
\p_tR^i_k\in L^\infty ([0,T_0], W^{5-2i,2}(\Om))\label{es-R-impr}
\end{align}
for $0\leq i\leq 2$.
\end{lem}

\begin{proof}
For simplicity, we decompose $R_k$ as
\[R_k=F_k^1+F_k^2,\]
where
\begin{align*}
F^1_k=&\sum_{\substack{a_1+\cdots, a_s+i+j=k,\\ 0\leq a_l,i,j\leq k-1}}v_{a_1}\#\cdots\#v_{a_s}\# \p v_i\#\p v_j,\\
F^2_k=&\sum_{\substack{b_1+b_2+b_3=k+1,\\ 1\leq b_l\leq k-1}}v_{b_1}\#v_{b_2}\#v_{b_3}.
\end{align*}
	
We focus our analysis on establishing estimates for $F_k^1$, as the corresponding bounds for $F^2_k$ can be derived using analogous arguments.
	
Now we consider the first and second time derivatives of $\p_t F_k^1$ and $\p^2_t F_k^1$. Through direct computation, we obtain
\begin{align*}
\p_tF^1_k=&\sum_{\substack{a_1+\cdots, a_s+i+j=k+1,\\ 0\leq a_l,i,j\leq k}}v_{a_1}\#\cdots\#v_{a_s}\# \p v_i\#\p v_j,\\
\p^2_tF^1_k=&\sum_{\substack{a_1+\cdots, a_s+i+j=k+2,\\ 0\leq a_l,i,j\leq k+1}}v_{a_1}\#\cdots\#v_{a_s}\# \p v_i\#\p v_j.
\end{align*}
	
Without loss of generality, we assume that $k\geq 2$, since $R_1=0$. We derive from the estimates \eqref{es-T-k} that
\[v_j\in L^\infty([0,T_0], W^{5,2}(\Om))\]
for $0\leq j\leq k-1$. Furthermore, from the bound \eqref{es-v-k-new}, we have
\[v_{k+i}\in L^\infty([0,T_0], W^{4-2i}(\Om))\]
for $0\leq i\leq 2$, which yields
\begin{align*}
v_k\in &L^\infty([0,T_0], W^{4,2}(\Om)), \,\,v_{k+1}\in L^\infty([0,T_0], W^{2,2}(\Om)),\\
\p v_k\in &L^\infty([0,T_0], W^{3,2}(\Om)),\,\,\p v_{k+1}\in L^\infty([0,T_0], W^{1,2}(\Om))
\end{align*}
	
Applying Lemma \ref{alg}, we then obtain
\[\p_t^iF_k^1\in L^\infty ([0,T_0], W^{5-2i,2}(\Om))\]
for $i=1,2$.
\end{proof}

\begin{lem}\label{equiv-W^{5,2}-es-om}
There exists a constant $C$ independent of $\ep$ such that the solution $\om_\ep$ for the problem \eqref{para-Nuemann-eq-v_1} obtained in Theorem \ref{W^{5,2}-solu-om-ep} satisfies
\begin{equation}\label{equiv-es-om-4}
\begin{aligned}
	\norm{\om_\ep}^2_{W^{5,2}}+\norm{\p_t \om_\ep}^2_{W^{3,2}}\leq &C(\norm{V}^2_{W^{1,2}}+\norm{\p_t \om_\ep}^2_{L^2}+\norm{\om_\ep}^2_{L^2}+1)\\
	\leq &C(\norm{\p V^{\top}}^2_{L^2}+1).
\end{aligned}
\end{equation}
\end{lem}
\begin{proof}
Using the formula \eqref{eq-De-om-ep}, a direct computation yields
\begin{align*}
	\p^3 \De \om_\ep=&\frac{1}{1+\ep^2}(\ep \p^3\p_t \om_\ep-u\times \p^3\p_t\om_\ep)\\
	&+\p^3u\#\p_t \om_\ep+\p^2 u\# \p\p_t \om_\ep+\p u\# \p^2 \p_t \om_\ep\\
	&+\ep\p^3(\om_\ep\#v_1\#u)+\p^3(\p u\#\p \om_\ep\# u+\De u\# \om_\ep\# u+\p u\#\p u\# \om_\ep)\\
	&+\p^3R_k.
\end{align*}
Then applying estimates \eqref{es-u-new} and \eqref{es-R1}, we obtain
\begin{align*}
\norm{\p^3\De \om_\ep}^2_{L^2}\leq& C\norm{\p^3 \p_t \om_\ep}^2_{L^2}+C\norm{u}^2_{W^{3,2}}\norm{\p_t\om_\ep}^2_{W^{2,2}}\\
&+C\ep(\norm{u}^2_{W^{5,2}}+\norm{v_1}^2_{W^{3,2}})^2\norm{\om_\ep}^2_{W^{4,2}}+C\norm{\p^3 R_k}^2_{L^2}\\
\leq&C(\norm{\p_t \om_\ep}^2_{W^{3,2}}+\norm{\om_\ep}^2_{W^{4,2}}+1).
\end{align*}
This leads to the key estimate for the $W^{5,2}$-norm of $\om_\ep$:
\begin{equation}\label{es-p^3-De-om}
	\begin{aligned}
	\norm{\om_\ep}^2_{W^{5,2}}\leq &C(\norm{\p_t \om_\ep}^2_{W^{3,2}}+\norm{\om_\ep}^2_{W^{4,2}}+1)\\
	\leq &C(\norm{\p_t \om_\ep}^2_{W^{3,2}}+1),
	\end{aligned}
\end{equation}
where the last inequality follows from the uniform \eqref{W^{4,2}-bound-om}.

On the other hand, from formula \eqref{eq-Dep_t-om}, we get
\begin{align*}
\p \De \p_t\om_\ep=&\frac{1}{1+\ep^2}(\ep \p V-u\times \p V)\\
&+\p u\# V+\p v_1\# \p_t \om_\ep+v_1\#\p \p_t \om_\ep+\ep\p\p_t(\om_\ep\#v_1\#u)\\
&+\p\p_t(\p u\#\p \om_\ep\# u+\De u\# \om_\ep\# u+\p u\#\p u\# \om_\ep)+\p\p_tR_k.
\end{align*}
Then we can derive from the estimates \eqref{es-u-new} and \eqref{es-T-k} that
\begin{align*}
\norm{\p\De \p_t\om_\ep}^2_{L^2}\leq &C(\norm{u}^2_{W^{3,2}}+1)\norm{V}^2_{W^{1,2}}+C(\norm{v_1}^2_{W^{1,2}}+1)^2\norm{\p_t \om_\ep}^2_{W^{1,2}}\\
&+C(\norm{u}^2_{W^{4,2}}+\norm{v_1}^2_{W^{5,2}})^2(\norm{\om_\ep}^2_{W^{4,2}}+\norm{\p_t \om_\ep}^2_{W^{2,2}})\\
&+C\ep\norm{v_2}^2_{W^{1,2}}\norm{\om_\ep}^2_{W^{2,2}}+\norm{\p_t R_k}^2_{W^{1,2}}\\
\leq &C(\norm{V}^2_{W^{1,2}}+\norm{\p_t \om_\ep}^2_{L^2}+\norm{\om_\ep}^2_{L^2}+1)\\
\leq& C(\norm{\p V^{\top}}^2_{L^2}+1),
\end{align*}
where we have used the uniform bound \eqref{W^{4,2}-bound-om} and the estimates \eqref{es-V^perp} for $V^{\perp}$. This immediately yields the desired $W^{3,2}$-estimate for $\p_t \om_\ep$:
\begin{align}
\norm{\p_t\om_\ep}^2_{W^{3,2}}\leq &C(\norm{\p V^{\top}}^2_{L^2}+1).\label{W^{3,2}-p_t-om}
\end{align}

Therefore, combining the above bounds \eqref{es-p^3-De-om} with \eqref{W^{3,2}-p_t-om} yields the desired estimate \eqref{equiv-es-om-4}.
\end{proof}

We now establish uniform $W^{5,2}$-energy estimates for $\om_\ep$. Taking $-\De V^\top_h$ as a test function for \eqref{para-Nuemann-eq-p-2-t-om'} and then passing to the limit as $h\to 0$, we obtain the following energy identity:
\begin{equation}\label{es-p-V-top}
\begin{aligned}
&\frac{1}{2}\(\int_{\Om}|\p V^\top|^2dx(T)-\int_{\Om}|\p V^\top|^2dx(0)\)\\
=&-\ep\int_0^T\int_{\Om}|\De V^\top|^2dxdt-\ep\int_0^T\int_{\Om}\<\De V^{\perp}, \De V^\top\>dxdt\\
&-\int_0^T\int_{\Om}\<u\times \De V^{\perp}, \De V^\top\>dxdt+\int_0^T\int_{\Om}\<\p(V\# v_1\# u), \p V^\top\>dxdt\\
&-\ep\int_0^T\int_{\Om}\<\mbox{div}(\<\p u, V\>u), \De V^\top\>dxdt-\ep\int_0^T\int_{\Om}\<\p u\# u\# \p V, \De V^\top\>dxdt\\
&-\ep\int_0^T\int_{\Om}\<\p u\#\p u\# V, \De V^\top\>dxdt-\ep\int_0^T\int_{\Om}\<\mbox{div}F, \De V^\top\>dxdt\\
&+\int_0^T\int_{\Om}\<\p \bar{R_1}, \p V^\top\>dxdt\\
=&-\ep\int_0^T\int_{\Om}|\De V^\top|^2dxdt+N_1+\cdots +N_8.
\end{aligned}
\end{equation}

Next, we provide detailed estimates for the terms $N_1$-$N_8$. We begin by estimating the terms $N_1$. Since the term $\De V^{\perp}$ is given by
\begin{align*}
\De V^\perp=&-\De\((\<v_2, \om_\ep\>+2\<v_1,\p_t\om_\ep\>)u\),
\end{align*}
we can make use of the estimates \eqref{es-u-new} to give a bound of $\norm{\De V^\perp}_{L^2}$:
\begin{align*}
\norm{\De V^{\perp}}^2_{L^2}\leq& C(\norm{v_2}^4_{W^{2,2}}+\norm{v_1}^4_{W^{2,2}}+\norm{u}^4_{W^{4,2}})(\norm{\om_\ep}^2_{W^{2,2}}+\norm{\p_t \om_\ep}^2_{W^{2,2}})\\
\leq& C(\norm{\om_\ep}^2_{W^{2,2}}+\norm{\p_t \om_\ep}^2_{W^{2,2}})\leq C.
\end{align*}
Hence, we have
\begin{align*}
|N_1|\leq &C\ep \int_0^T\int_{\Om}|\De V^{\perp}|^2dxdt+\frac{\ep}{10}\int_0^T\int_{\Om}|\De V^\top|^2dxdt\\
\leq &C\ep\int_0^T(\norm{\om_\ep}^2_{W^{2,2}}+\norm{\p_t \om_\ep}^2_{W^{2,2}})dx+\frac{\ep}{10}\int_0^T\int_{\Om}|\De V^\top|^2dxdt\\
\leq &C\ep T+\frac{\ep}{10}\int_0^T\int_{\Om}|\De V^\top|^2dxdt.
\end{align*}

On the other hand, the term $u\times \De V^\perp$ can be expanded as
\begin{align*}
u\times \De V^\perp=&-2\p (\<v_2, \om_\ep\>+2\<v_1,\p_t\om_\ep\>) u\times \p u\\
&-(\<v_2, \om_\ep\>+2\<v_1,\p_t\om_\ep\>) v_1.
\end{align*}
Using estimates \eqref{es-u-new} and \eqref{W^{4,2}-bound-om}, we derive the following pointwise bound
\[|\p(u\times \De V^\perp)|\leq C(|\p^2v_2|+|\p v_2|+|\p^2 \p_t \om_\ep|+|\p \p_t \om_\ep|+1).\]
This leads to the estimate for $N_2$
\begin{align*}
|N_2|=&\left|\int_0^T\int_{\Om}\<\p(u\times \De V^{\perp}), \p V^\top\>dxdt\right|\\
=&\int_0^T\int_{\Om}|\p V^\top|^2dxdt+CT,
\end{align*}
where we have used the following boundary condition:
\[\frac{\p V^\top}{\p \nu}|_{\p \Om\times [0,T_0]}=\frac{\p V}{\p \nu}|_{\p \Om\times [0,T_0]}-\frac{\p }{\p \nu}\<V,u\>u|_{\p \Om\times [0,T_0]}=0.\]

Next, we turn to the estimation of the remaining terms $N_3$-$N_6$. Applying the a priori estimates \eqref{es-u-new} and \eqref{W^{4,2}-bound-om}, we proceed as follows:
\begin{align*}
|N_3|\leq &C\int_0^T\int_{\Om}(|\p V|+|V|)|\p V^\top|dxdt\\
\leq &C \int_0^T\int_{\Om}|\p V^\top|^2dxdt+C\int_0^T\int_{\Om}|\p V^\perp|^2+|V|^2dxdt\\
\leq &C \int_0^T\int_{\Om}|\p V^\top|^2dxdt+CT.\\
|N_4+N_5+N_6|\leq &C\ep\int_0^T\int_{\Om}(|\p V|+|V| )|\De V^\top|dxdt\\
\leq &C\ep \int_0^T\int_{\Om}|\p V^\top|^2dxdt+C\int_0^T\int_{\Om}|\p V^\perp|^2dxdt\\
&+C\int_0^T\int_{\Om}|V|^2dxdt+\frac{\ep}{10}\int_0^T\int_{\Om}|\De V^\top|^2dxdt\\
\leq &C\ep\int_0^T\int_{\Om}|\p V^\top|^2dxdt+C\ep T+\frac{\ep}{10}\int_0^T\int_{\Om}|\De V^\top|^2dxdt.
\end{align*}

For the term $N_7$, we have
\begin{align*}
|N_7|\leq &C\ep\int_0^T\int_{\Om}|\mbox{div}F|^2dxdt+\frac{\ep}{10}\int_0^T\int_{\Om}|\De V^\top|^2dxdt\\
\leq &C\ep T+\frac{\ep}{10}\int_0^T\int_{\Om}|\De V^\top|^2dxdt,
\end{align*}
where we used the estimates \eqref{es-u-new}, \eqref{W^{4,2}-bound-om} to give a pointwise bound of $\mbox{div}F$:
\begin{align*}
|\mbox{div}F|\leq C(|\p^2v_2|+|\p v_2|+|\p \p_t \om_\ep|+1).
\end{align*}

The last term $N_8$ admits the following estimate:
\begin{align*}
|N_8|\leq &C\int_0^T\int_{\Om}|\p \bar{R}_1|^2dxdt+C\int_0^T\int_{\Om}|\p V^\top|^2dxdt\\
\leq & C\int_0^T\norm{\p_t\om_\ep}^2_{W^{3,2}}dt+C\int_0^T\int_{\Om}(|\p V|^2+|\p V^\top|^2dxdt+CT\\
\leq &C\int_0^T\int_{\Om}|\p V^\top|^2dxdt+CT.
\end{align*}
Here we applied the estimates \eqref{es-T-k}, \eqref{es-u-new}, and \eqref{es-R-impr} to obtain a bound of $|\p \bar{R_1}|$:
\begin{align*}
	|\p \bar{R}_1|\leq &C(|\p^2v_2|+|\p v_2|)+C(|\p \p^2_t R_k|+|\p^2_t R_k|)\\
	&+C(|\p^2 \p_t \om_\ep|+|\p\p_t \om_\ep|+|\p \De \om_\ep|)\\
	&+C(|\p \De \p_t \om_\ep|+|\De \p_t \om_\ep|+|\p V|+|V|),
\end{align*}
and additionally used the equivalent bound \eqref{equiv-es-om-4} to simplify the final estimate.

Therefore, we substitute the above estimates for $N_1$-$N_8$ into \eqref{es-p-V-top} to deduce
\begin{align*}
&\frac{1}{2}\(\int_{\Om}|\p V^\top|^2dx(T)-\int_{\Om}|\p V^\top|^2dx(0)\)\leq C\int_0^T\int_{\Om}|\p V^\top|^2dxdt+CT.
\end{align*}
Gronwall's inequality yields the uniform bound:
\begin{align*}
\sup_{0\leq t\leq T}\norm{\p V^\top}^2_{L^2}\leq C(T_0, \norm{\p V^\top(0)}^2_{L^2})\leq C(T_0, \norm{u_0}_{W^{2k+3,2}}).
\end{align*}
This estimate further implies that
\begin{align}
\sup_{0\leq t\leq T_0}\sum_{i=0}^2\norm{\p^i_t\om_\ep}^2_{W^{5-2i,2}}\leq C.\label{W^{5,2}-bound-om}
\end{align}

By virtue of the uniform energy bound \eqref{W^{5,2}-bound-om} satisfied by $\om_\ep$, we apply Lemma \ref{A-S} (Aubin-Simon compactness theorem) to deduce that the limiting map $v_k$ of $\om_\ep$ inherits the following regularity estimates
\begin{align}
\sup_{0\leq t\leq T_0}\sum_{i=0}^2\norm{\p_t^iv_k}^2_{W^{5-2i,2}(\Om)}\leq C.\label{impro-es-v_k}
\end{align}

\subsection{The proof of the property $\T_{k+1}$}
Assume that $u_0$ satisfies the $(k+1)$-th order compatibility conditions \eqref{com-cond-1}, and that the solution $u$ possesses property $\T_k$ ( see~~\eqref{es-case-k}), namely for $0\leq i\leq k+1$, we have
\begin{align}
	v_i\in L^\infty([0,T_0],W^{2k+3-2i,2}(\Om)),\label{es-case-k1}
\end{align}
Additionally, from the improved estimates \eqref{impro-es-v_k} for $v_k$, we obtain the improved regularity
\begin{align}
	v_{i}\in L^\infty([0,T_0], W^{2(k+1)+3-2i,2}(\Om))\label{impro-es-v_k-1}
\end{align}
for $k\leq i\leq k+2$.

With these regularity results established, we now proceed to demonstrate that the solution $u$ satisfies the property $\T_{k+1}$.

\begin{prop}\label{pf-T_{k+1}}
Assume that $u_0$ satisfies the $(k+1)$-th order compatibility conditions \eqref{com-cond-1}, and the solution $u$ satisfies the bounds \eqref{es-case-k1}. Then for $0\leq n\leq k+2$, we have
\begin{align}
v_{k+2-n}\in L^\infty([0,T_0],W^{2n+1,2}(\Om)),\label{es-case-(k+1)}
\end{align}
which implies that $u$ satisfies property $\T_{k+1}$
\end{prop}
\begin{proof}
We employ mathematical induction to prove this proposition. First, observe that for $n=0,1,2$, we have already established
\[v_{k+2-n}\in L^\infty([0,T_0], W^{2n+1,2}(\Om)).\]
	
Now we assume that \eqref{es-case-(k+1)} holds for $n=l$ with $2\leq l\leq k+1$. Then we prove the estimates \eqref{es-case-(k+1)} for $n=l+1$.
	
Under the induction hypothesis, we have
\begin{itemize}
\item For $0\leq n\leq l$,
\[v_{k+2-n}\in L^\infty([0,T], W^{2n+1}(\Om)).\]
\item For $0\leq j\leq k+1-l$,
\[v_{j}\in L^\infty([0,T], W^{2k+3-2j}(\Om)).\]
\end{itemize}
	
From equation \eqref{eq-v_k}, we derive the following expression for $\De v_{k+1-l}$:
\begin{align*}
\De v_{k+1-l}=&-u\times v_{k+2-l}+\p u\# \p v_{k+1-l}\# u\\
&+\De u\# v_{k+1-l}\# u+\p u\#\p u\# v_{k+1-l}\\
&+\sum_{\substack{a_1+\cdots, a_s+i+j=k+1-l,\\ 1\leq a_l,i,j\leq k-l}}v_{a_1}\#\cdots\#v_{a_s}\# \p v_i\#\p v_j\\
&+\sum_{\substack{b_1+b_2+b_3=k+2-l,\\ 1\leq b_l\leq k-l}}v_{b_1}\#(u\times v_{b_2})\#v_{b_3}.
\end{align*}

Since $v_{k+2-l}\in L^\infty([0,T_0], W^{2l+1}(\Om))$ and $v_i\in  L^\infty([0,T_0], W^{2l+1}(\Om))$ for $0\leq i\leq k+1-l$ with $2\leq l\leq k+1$, we apply Lemma \ref{alg} to conclude
\[\De v_{k+1-l}\in L^\infty([0,T_0], W^{2l}(\Om)),\]
which implies that
\begin{align}
v_{k+1-l}\in L^\infty([0,T_0], W^{2(l+1)}(\Om)).\label{es-v_{k+1-l}}
\end{align}

Furthermore, since $v_i\in  L^\infty([0,T_0], W^{2(l+1)+1}(\Om))$ for $i\leq k-l$ with $2\leq l\leq k$, the estimate \eqref{es-v_{k+1-l}} again yields the improved estimate
\[\De v_{k+1-l}\in L^\infty([0,T_0], W^{2l+1}(\Om)),\]
which implies
\begin{align}
v_{k+1-l}\in L^\infty([0,T_0], W^{2(l+1)+1}(\Om)).\label{es-v_{k+1-l}-1}
\end{align}
This completes the induction step. In particular, the special case $n=k+1$ in \eqref{es-v_{k+1-l}-1} gives
\[v_1\in L^\infty([0,T_0], W^{2(k+1)+1}(\Om)).\]

It remains to establish $ u\in L^\infty([0,T_0], W^{2(k+1)+3}(\Om))$. Recall that
\begin{align*}
\De u=&-u\times v_1-|\p u|^2u.
\end{align*}
Since $v_1\in L^\infty([0,T_0], W^{2k+3}(\Om))$ and $u\in L^\infty([0,T_0], W^{2k+3}(\Om))$, we deduce from the above equation that
\[\De u\in L^\infty([0,T_0], W^{2(k+1)}(\Om)),\]
which yields
\[u\in L^\infty([0,T_0], W^{2(k+1)+2}(\Om)).\]
This estimate further implies
\[\De u\in L^\infty([0,T_0], W^{2(k+1)+1}(\Om)).\]
We then apply $L^2$-estimates for Laplacian operator to obtain
\[u\in L^\infty([0,T_0], W^{2(k+1)+3}(\Om)).\]

Therefore, the proof of estimates \eqref{es-case-(k+1)} is completed. By Lemma \ref{equiv-es-p-t-u}, the regularity bounds \eqref{es-case-(k+1)} in fact imply the property $\T_{k+1}$.
\end{proof}

Finally, we prove the main theorem \ref{main-thm}.

\begin{proof}[\bf{The proof of Theorem \ref{main-thm}}]
In order to establish this theorem, it is sufficient to demonstrate that the property $\T_k$ (namely, \eqref{es-case-k}) holds for each $k\geq 1$. We proceed by induction on $k$. The property $\T_1$ has been established in our previous work \cite{CW4}. Now assuming that $u$ satisfies the property $\T_k$ with $k\geq 1$, Proposition \ref{pf-T_{k+1}} guarantees that $\T_{k+1}$ holds.

Additionally, if $u_0\in C^\infty(\bar{\Om})$, which satisfies the $k$-order compatibility conditions defined by \eqref{com-cond-1} for any $k\geq 0$, the property $\T_k$ yields that
\[\sup_{0<t<T_0}\norm{\p^j_t\p^s_xu}^2_{L^2(\Om}<\infty\]
for any $j,s\in \mathbb{N}$. So, it follows from the Sobolev embedding theorem that
\[u\in C^\infty(\bar{\Om}\times[0,T_0]),\]

Therefore, the proof is completed.
\end{proof}

\medskip
\appendix
\renewcommand{\appendixname}{Appendix~\Alph{section}}
\section{Local existence of regular solutions to parabolic equations}\label{s: parabolic-eq}
In this appendix, we consider the following initial-Neumann boundary value problem:
\begin{equation}\label{eq-h}
\begin{cases}
\p_t h+\<\p_t u, h\>u=(\ep I+u\times)(\De h+f_1\#\p h+f_2\#h)+f_3\\[1ex]
\frac{\p h}{\p \nu}|_{\p\Om\times [0,T]}=0,\\[1ex]
h(0)=h_0: \Om\to \Real^3,\quad \frac{\p h_0}{\p \nu}|_{\p\Om}=0,
\end{cases}
\end{equation}
where
\begin{itemize}
	\item $u:\Om\times [0,T_0]\to \mathbb{S}^2$ is a given map with Neumann boundary condition $\frac{\p u}{\p \nu}|_{\p\Om}=0$;
	\item For $i=1,2,3$, $f_i:\Om\times [0,T_0]\to \Real^3$ are given vector fields;
	\item The notation $\#$ denotes the linear contraction.
\end{itemize}

\subsection{Galekin approximation and $W^{2,2}$-regular solutions}\label{ss: Ga-approx}
\medskip
Let $\Om$ be a smooth bounded domain in $\Real^{3}$, $\ld_{i}$ be the $i^{th}$ eigenvalue of the operator $\De-I$ with Neumann boundary conditions. The corresponding eigenfunction $g_{i}$ satisfies
\[(\De-I)g_{i}=-\ld_{i}g_{i}\,\,\,\,\quad\text{with}\quad\,\,\,\,\frac{\p g_{i}}{\p \nu}|_{\p\Om}=0.\]

Without loss of generality, we assume that $\{g_{i}\}_{i=1}^{\infty}$ forms a complete orthonormal basis of $L^{2}(\Om,\Real^{1})$. For each $n\in \mathbb{N}$, define the finite-dimensional subspace
$$H_{n}=\text{span}\{g_{1},\dots g_{n}\}\subset L^2(\Om),$$
and let $P_{n}:L^{2}\to H_{n}$ denote the Galerkin projection given by
\[f^{n}=P_{n}f=\sum_{1}^{n}\<f,g_{i}\>_{L^{2}}g_{i} \quad \text{for any} \,\,f\in L^{2}.\]
The following  result is established in \cite{CJ}.

\begin{lem}\label{es-P_n}
There exists a constant $C$, independent of $n$, such that the projection $P_n$ satisfies the following properties.
	\begin{enumerate}
		\item For $f\in W^{1,2}(\Om,\Real^1)$,
		$$\norm{P_n(f)}_{W^{1,2}(\Om)}\leq \norm{f}_{W^{1,2}(\Om)},$$
		\item For $f\in W^{2,2}(\Om,\Real^1)$ with $\frac{\p f}{\p \nu}|_{\p \Om}=0$, $$\norm{P_n(f)}_{W^{2,2}(\Om)}\leq C\norm{f}_{W^{2,2}(\Om)},$$
		\item For $f\in W^{3,2}(\Om,\Real^1)$ with $\frac{\p f}{\p \nu}|_{\p \Om}=0$, $$\norm{P_n(f)}_{W^{3,2}(\Om)}\leq C\norm{f}_{W^{3,2}(\Om)}.$$
	\end{enumerate}
\end{lem}

Now, we consider the Galerkin approximation of \eqref{eq-h}, given by the following system:
\begin{equation}\label{eq-Ga-h}
\begin{cases}
\p_t h_n +P_n(\<\p_t u, h_n\>u)=P_n\{(\ep I+u\times)(\De h_n+f_1\#\p h_n+f_2\#h_n)+f_3\}\\[1ex]	
h_n(0)=P_n(h_0).
\end{cases}
\end{equation}
By the standard theory of ordinary differential equations, there exists a unique solution $h_n$ defined on a maximal interval of existence $[0,T_n)$, if $u$ and $f_i$ with $i=1,2,3$ satisfy the following regularity conditions:
\begin{equation}\label{assumption-A}
\begin{aligned}
u\in& C^0([0,T_0], W^{3,2}(\Om))\text{and}\,\, \p_t u\in C^0([0,T_0],W^{1,2}(\Om)),\\
f_i\in& C^0([0,T_0], W^{1,2}(\Om))\,\,\text{for $i=1,2$},\\
f_3\in& C^0([0,T_0], L^2(\Om))\cap L^2([0,T_0], W^{1,2}(\Om)).
\end{aligned}
\end{equation}

Next, we derive uniform $W^{2,2}$-energy estimates for the approximate solutions $h_n$, independent of $n$.  Using $h_n$ to test equation \eqref{eq-Ga-h} gives
\begin{equation}\label{es-h}
\begin{aligned}
\frac{1}{2}\p_t\int_{\Om}|h_n|^2dx\leq &-\ep\int_{\Om}|\p h_n|^2dx+C\int_{\Om}(|\p u|+|f_1|)|\p h_n||h_n|dx\\
&+C\int_{\Om}(|f_2|+|\p_t u|)|h_n|^2dx+C\int_{\Om}|f_3||h_n|dx\\
\leq &-\ep\int_{\Om}|\p h_n|^2dx\\
&+C_\ep(\norm{|\p u|}_{L^4}+\norm{f_1}_{L^4})\norm{h_n}_{L^4}\norm{\p h_n}_{L^2}\\
&+C(\norm{f_2}_{L^2}+\norm{\p_t u}_{L^2})\norm{h_n}^2_{L^4}+C\norm{f_3}_{L^2}\norm{h_n}_{L^2}\\
\leq & -\frac{\ep}{2}\int_{\Om}|\p h_n|^2dx+C\norm{f_3}^2_{L^2}\\
&+C_\ep(\norm{u}^2_{W^{2,2}}+\norm{f_1}^2_{W^{1,2}}+\norm{f_2}_{L^2}^2+\norm{\p_t u}^2_{L^2}+1)^2\norm{h_n}^2_{L^2},
\end{aligned}
\end{equation}
where we applied the following interpolation inequality:
\[\norm{h_n}_{L^4}\leq C\norm{h_n}^{3/4}_{W^{1,2}}\norm{h_n}_{L^2}^{1/4}.\]

Then we take $\De^2 h_n$ as a test function for \eqref{eq-Ga-h}. A direct computation yields
\begin{equation}\label{es-h1}
\begin{aligned}
\frac{1}{2}\p_t\int_{\Om}|\De h_n|^2dx
\leq&-\ep\int_{\Om}|\p \De h_n|^2dx+C\int_{\Om}|\p u||\De h_n||\p \De h_n|dx\\
&+C\int_{\Om}|\p u||f_1||\p h_n||\p \De h_n|dx\\
&+C\int_{\Om}(|\p f_1||\p h_n|+|f_1||\p^2 h_n|)|\p \De h_n|dx\\
&+C\int_{\Om}|\p u|(|f_2|+|\p_t u|)|h_n||\p \De h_n|dx\\
&+C\int_{\Om}((|\p f_2|+|\p\p_t u|)|h_n|+(|f_2|+\p_t u|)|\p h_n|)|\p \De h_n|dx\\
&+C\int_{\Om}(|\p f_3||\p \De h_n|dx\\
=&-\ep\int_{\Om}|\p \De h_n|^2dx\\
&+I_1+I_2+I_3+I_4+I_5+I_6+I_7.
\end{aligned}
\end{equation}
The terms $I_1$-$I_6$ admit the following estimates:
\begin{align*}
|I_1|\leq& C\norm{\p u}_{L^6}\norm{\De h_n}_{L^3}\norm{\p\De h_n}_{L^2}\\
\leq &C\norm{u}_{W^{1,2}}\norm{\De h_n}^{1/2}_{L^2}\norm{\De h_n}_{W^{1,2}}^{1/2}\norm{\p \De h_n}_{L^2}\\
\leq &C_\ep(\norm{u}_{W^{1,2}}^4+1)\norm{\De h_n}^2_{L^2}+\frac{\ep}{14}\norm{\p \De h_n}^2_{L^2},\\
|I_2|\leq &C\norm{\p u}_{L^6}\norm{f_1}_{L^6}\norm{\p h_n}_{L^6}\norm{\p\De h_n}_{L^2}\\
\leq &C_\ep\norm{u}_{W^{1,2}}^2(\norm{\De h_n}^2_{L^2}+\norm{h_n}^2_{L^2})+\frac{\ep}{14}\norm{\p \De h_n}^2_{L^2},\\
|I_3|\leq &C\norm{\p f_1}_{L^2}\norm{\p h_n}_{L^\infty}\norm{\p\De h_n}_{L^2}+C\norm{\p f_1}_{L^6}\norm{\p^2 h_n}_{L^3}\norm{\p\De h_n}_{L^2}\\
\leq &C\norm{\p f_1}_{L^2}\norm{\p h_n}^{1/4}_{L^2}\norm{\p h_n}^{3/4}_{W^{2,2}}\norm{\p\De h_n}_{L^2}\\
&+C\norm{\p f_1}_{L^6}\norm{\p^2 h_n}^{1/2}_{L^2}\norm{\p^2 h_n}^{1/2}\norm{\p\De h_n}_{L^2}\\
\leq &C(\norm{f_1}^8_{W^{1,2}}+1)(\norm{\De h_n}^2_{L^2}+\norm{h_n}^2_{L^2})+\frac{\ep}{14}\norm{\p \De h_n}^2_{L^2},\\
|I_4|\leq &C\norm{\p u}_{L^6}(\norm{f_2}_{L^6}+\norm{\p_t u}_{L^6})\norm{h_n}_{L^6}\norm{\p\De h_n}_{L^2}\\
\leq&C_\ep \norm{u}^2_{W^{1,2}}(\norm{f_2}^2_{W^{1,2}}+\norm{\p_t u}^2_{W^{1,2}})(\norm{\De h_n}^2_{L^2}+\norm{h_n}^2_{L^2})+\frac{\ep}{14}\norm{\p \De h_n}^2_{L^2},\\
|I_5|\leq &C(\norm{f_2}_{W^{1,2}}+\norm{\p_t u}_{W^{1,2}})\norm{h_n}_{W^{2,2}}\norm{\p\De h_n}_{L^2}\\
\leq&C_\ep(\norm{f_2}^2_{W^{1,2}}+\norm{\p_t u}^2_{W^{1,2}})(\norm{\De h_n}^2_{L^2}+\norm{h_n}^2_{L^2})+\frac{\ep}{14}\norm{\p \De h_n}^2_{L^2},\\
|I_6|\leq &C(\norm{u}_{W^{1,2}}+1)\norm{f_3}_{W^{1,2}}\norm{\p \De h_n}_{L^2}\\
\leq&C_\ep (\norm{u}^2_{W^{1,2}}+1)\norm{f_3}^2_{W^{1,2}}+\frac{\ep}{14}\norm{\p \De h_n}^2_{L^2}.
\end{align*}

Substituting the above estimates of $I_1$-$I_7$ into the formula \eqref{es-h1}, we obtain
\begin{equation}\label{es-h2}
\begin{aligned}
	\frac{1}{2}\p_t\int_{\Om}|\De h_n|^2dx\leq&-\frac{\ep}{2}\int_{\Om}|\p \De h_n|^2dx+C_\ep \norm{f_3}^2_{W^{1,2}}\\
	&+C_\ep(\norm{h_n}^2_{L^2}+\norm{\De h_n}^2_{L^2}),
\end{aligned}
\end{equation}
where we have used the regularity assumptions \eqref{assumption-A}.

Then combining the estimates \eqref{es-h} with \eqref{es-h2}, we derive from Gronwall's inequality that
\begin{align}
\sup_{0\leq t\leq T_0}\norm{h_n}^2_{W^{2,2}(\Om)}+\norm{h_n}^2_{L^2([0,T_0], W^{3,2}(\Om)}\leq C_\ep(T_0, \norm{h_0}_{W^{2,2}(\Om)}).\label{es-h3}
\end{align}
Passing to the limit as $n \to \infty$, we obtain a solution $h$ to the original problem \eqref{eq-h} that satisfies the regularity
\begin{align}
h\in L^\infty([0,T_0],W^{2,2}(\Om))\cap L^2([0,T_0], W^{3,2}(\Om)).\label{es-h4}
\end{align}

Next, we establish regularity estimates for $\p_t h$. From the evolution equation \eqref{eq-h}, that is,
\[\p_t h=(\ep I+u\times)(\De h+f_1\#\p h+f_2\#h)+f_3-\<h,\p_t u\>u,\]
we infer from the bound \eqref{es-h4} together with the regularity assumptions \eqref{assumption-A} that
\[\p_th\in L^\infty([0,T_0],L^2(\Om))\cap L^2([0,T_0], W^{1,2}(\Om)).\]
Applying Lemma~\ref{C^0-em}, it follows that
\[h\in C^0([0,T_0],W^{2,2}(\Om)).\]
Consequently, using again the regularity assumptions \eqref{assumption-A} and the equation \eqref{eq-h}, we further obtain
\[\p_t h\in C^0([0,T_0],L^2(\Om)).\]

Now we summaries these regularity estimates in the following theorem.
\begin{thm}\label{thm-h}
Let $\Om \subset \mathbb{R}^3$ be a bounded domain with smooth boundary, and let $h_0 \in W^{2,2}(\Om)$ satisfy the compatibility condition
$$\frac{\p h_0}{\p \nu}|_{\p\Om}=0.$$
Under the regularity assumptions \eqref{assumption-A}, the initial-Neumann boundary value problem \eqref{eq-h} admits a unique solution such that
\[\p^i_th\in C^0([0,T_0],W^{2-2i,2}(\Om))\cap L^2([0,T_0], W^{3-2i,2}(\Om))\]
for $i=0,1$.
\end{thm}

\subsection{$W^{3,2}$-regular solution}\label{ss: Ga-approx1}
In this part, under the following enhanced regularity assumptions:
\begin{equation}\label{assumption-A1}
\begin{aligned}
&\p^i_tu\in L^\infty([0,T_0], W^{4-2i,2}(\Om))\,\,\text{for}\,\,i=0,1,2,\\
&f_i\in L^\infty([0,T_0], W^{1,2}(\Om))\cap L^2([0,T_0],W^{2,2}(\Om))\,\,\text{for}\, i=1,2, \\
&\p_t f_i\in L^\infty([0,T_0], L^2(\Om)),\,\,\text{for}\, i=1,2, \\
&f_3\in L^\infty([0,T_0], W^{1,2}(\Om))\,\,\text{and}\,\,\p_t f_3\in L^2([0,T_0],L^2(\Om)),
\end{aligned}
\end{equation}
we establish $W^{3,2}$-regularity for the solution $h$ given in Theorem \ref{thm-h}. To this end, we analyze the equation satisfied by $\p_t h_n$:
\begin{equation}\label{eq-p_t-h_n}
\begin{aligned}
	\p^2_t h_n=&P_n\{(\ep I+ u\times )\De \p_t h_n+\p_t u\times (\De h_n+f_1\#\p h_n+f_2\#h_n)\}\\
	&+P_n\{(\ep I+ u\times )(\p_t f_1\#\p h_n+f_1\#\p \p_th_n)\}\\
	&+P_n\{(\ep I+ u\times )(\p_t f_2\#h_n+f_2\#\p_th_n)+\p_t f_3\}\\
    &-P_n\{\<\p_t h_n, \p_t u\>u+\<h_n, \p^2_t u\>u+\<h_n, \p_t u\>\p_tu\}.
\end{aligned}
\end{equation}

Taking $-\De \p_t h_n$ as a test function for \eqref{eq-p_t-h_n}, we obtain
\begin{equation}\label{es-p-t-h}
\begin{aligned}
\frac{1}{2}\p_t \int_{\Om}|\p \p_t h_n|^2dx\leq &-\ep\int_{\Om}|\De \p_t h_n|^2dx\\
	&+\int_{\Om}|\p_t u\times (\De h_n+f_1\#\p h_n+f_2\#h_n)||\De \p_t h_n|dx\\
	&+\int_{\Om}|\p_t f_1\#\p h_n+f_1\#\p \p_th_n||\De \p_t h_n|dx\\
	&+\int_{\Om}(|\p_t f_2\#h_n|+|f_2\#\p_th_n|+|\p_t f_3|)|\De \p_t h_n|dx\\
    &+\int_{\Om}(|\p_t u|^2|h_n|+|\p^2_t u||h_n|+|\p_tu||\p_th_n|)|\De \p_t h_n|dx\\
	=&-\ep\int_{\Om}|\De \p_t h_n|^2dx+II_1+II_2+II_3+II_4.
\end{aligned}
\end{equation}
Here, the terms $II_1$-$II_4$ admit the following estimates:
\begin{align*}
II_1\leq &C\norm{\p_t u}_{L^6}\norm{\De h_n}_{L^3}\norm{\De \p_t h_n}_{L^2}\\
&+\norm{\p_t u}_{L^6}\norm{f_1}_{L^6}\norm{\p h_n}_{L^6}\norm{\De \p_t h_n}_{L^2}\\
&+C\norm{\p_t u}_{L^6}\norm{f_2}_{L^3}\norm{h_n}_{L^\infty}\norm{\De \p_t h_n}_{L^2}\\
\leq &C_\ep \norm{\p_t u}^2_{W^{1,2}}(\norm{f_1}^2_{W^{1,2}}+\norm{f_2}^2_{W^{1,2}})\norm{h_n}^2_{W^{3,2}}+\frac{\ep}{6}\norm{\De \p_t h_n}^2_{L^2}.\\
II_2\leq &C\norm{\p_t f_1}_{L^2}\norm{\p h_n}_{L^\infty}\norm{\De \p_t h_n}_{L^2}\\
&+C\norm{f_1}_{L^6}\norm{\p \p_t h_n}_{L^3}\norm{\De \p_t h_n}_{L^2}\\
\leq &C_\ep \norm{\p_t f_1}^2_{L^2}\norm{h_n}^2_{W^{3,2}}+C_\ep(\norm{f_1}^4_{W^{1,2}}+1)\norm{\p \p_t h_n}^2_{L^2}+\frac{\ep}{6}\norm{\De \p_t h_n}^2_{L^2},\\
II_3\leq &C\norm{\p_t f_2}_{L^2}\norm{h_n}_{L^\infty}\norm{\De \p_t h_n}_{L^2}+C\norm{f_2}_{L^6}\norm{\p_t h_n}_{L^3}\norm{\De \p_t h_n}_{L^2}\\
&+C\norm{\p_t f_3}_{L^2}\norm{\De \p_t h_n}_{L^2}\\
\leq &C_\ep\norm{\p_t f_2}^2_{L^2}\norm{h_n}^2_{W^{2,2}}+C_\ep\norm{\p_tf_3}^2_{L^2}\\
&+C_\ep \norm{f_2}^2_{W^{1,2}}\norm{\p_t h_n}^2_{W^{1,2}}+\frac{\ep}{6}\norm{\De \p_t h_n}^2_{L^2},\\
II_4\leq &C(\norm{\p_t u}^2_{L^4}+\norm{\p^2_tu}_{L^2})\norm{h_n}_{L^\infty}\norm{\De \p_t h_n}_{L^2}+C\norm{\p_t u}_{L^6}\norm{\p_t h_n}_{L^3}\norm{\De \p_t h_n}_{L^2}\\
\leq &C_\ep(\norm{\p_t u}^4_{W^{1,2}}+\norm{\p^2_tu}^2_{L^2})\norm{h_n}^2_{W^{2,2}}+C_\ep \norm{\p_t u}^2_{W^{1,2}}\norm{\p_t h_n}^2_{W^{1,2}}+\frac{\ep}{6}\norm{\De \p_t h_n}^2_{L^2},\\
\end{align*}

Substituting the estimates of $II_1$-$II_3$ into \eqref{es-p-t-h}, we deduce that
\begin{equation}\label{es-p-t-h1}
\begin{aligned}
\frac{1}{2}\p_t \int_{\Om}|\p \p_t h_n|^2dx
\leq &-\frac{\ep}{2}\int_{\Om}|\De \p_t h_n|^2dx+C_\ep\norm{\p_t h_n}^2_{W^{1,2}}\\
&+C_\ep(\norm{h_n}^2_{W^{3,2}}+\norm{\p_tf_3}^2_{L^2}+1),
\end{aligned}
\end{equation}
where we used the regularity assumption \eqref{assumption-A1}.

From the Galerkin equation \eqref{eq-Ga-h} and the assumption \eqref{assumption-A1}, we also establish
\begin{align}
\sup_{0\leq t\leq T_0}\norm{\p_t h_n}^2_{L^2(\Om)}\leq C(\norm{h_n}^2_{W^{2,2}(\Om)}+1)\leq C_\ep(T_0, \norm{h_0}_{W^{2,2}(\Om)}).\label{es-p-t-h2}
\end{align}

Substituting this estimate \eqref{es-p-t-h2} into \eqref{es-p-t-h1}, we can apply Gronwall's inequality to show that
\begin{align}
\sup_{0\leq t\leq T_0}\norm{\p_t h_n}^2_{W^{1,2}}+\ep\int^{T_0}_0\norm{\p_t h_n}^2_{W^{2,2}(\Om)}dt\leq C(\ep, T_0, \norm{h_0}_{W^{3,2}(\Om)}),\label{es-p-t-h3}
\end{align}
where we used the initial bound:
\[\norm{\p \p_t h_n}^2_{L^2(\Om)}|_{t=0}\leq C\norm{P_n(h_0)}^2_{W^{3,2}(\Om)}\leq C\norm{h_0}^2_{W^{3,2}(\Om)},\]
following from the compatibility condition $\frac{\p h_0}{\p \nu}|_{\p\Om}=0$.

Furthermore, from equation \eqref{eq-p_t-h_n} and estimate \eqref{es-h4}, we also obtain
\[\int^{T_0}_0\norm{\p^2_t h_n}^2_{L^2}dt\leq C(\ep, T_0, \norm{h_0}_{W^{3,2}(\Om)}).\]

Therefore, passing to the limit $n\to \infty$ yields the time regularity for $h$:
\begin{equation}\label{es-h5}
\begin{aligned}
	\p_t h\in& C^0([0,T_0],W^{1,2}(\Om))\cap L^2([0,T_0],W^{2,2}(\Om)),\\
	\p^2_t h\in &L^2([0,T_0],L^2(\Om)).
\end{aligned}
\end{equation}

The estimate \eqref{es-h5} further implies improved regularity of $h$. Using equation \eqref{eq-h}, we derive the following elliptic equation
\begin{align}
\ep\De h+u\times \De h=-\p_t h+(\ep I+u\times)(f_1\# \p h+f_2\# h)+f_3.\label{elliptic-eq-h}
\end{align}
Under the additional assumption $f_i \in L^2([0,T_0],W^{2,2}(\Om))$ for $i=1,2,3$, the estimates \eqref{es-h4} and \eqref{es-h5} imply
\[\ep\De h+u\times \De h\in L^\infty([0,T_0],W^{1,2}(\Om))\cap L^2([0,T_0],W^{2,2}(\Om)),\]
This elliptic regularity produces the spatial improvement for $h$:
\[\norm{\De h}^2_{L^\infty([0,T_0],W^{1,2}(\Om))}+\norm{\De h}^2_{L^2([0,T_0],W^{2,2}(\Om))}<\infty,\]
which implies the following regularity estimate:
\[h\in C^0([0,T_0],W^{3,2}(\Om))\cap L^2([0,T_0], W^{4,2}(\Om)).\]

In conclusion, we obtain the following result.
\begin{thm}\label{thm-h1}
Let $\Om \subset \mathbb{R}^3$ be a bounded domain with smooth boundary, and let $h_0 \in W^{3,2}(\Om)$ satisfy the compatibility condition:
$$\frac{\p h_0}{\p \nu}|_{\p\Om}=0.$$
Under the assumption of regularity \eqref{assumption-A1}, the solution $h$ to problem \eqref{eq-h} established in Theorem \ref{thm-h} admits the enhanced regularity:
\begin{equation}\label{es-h7}
\begin{aligned}
	\p^i_th\in &C^0([0,T_0],W^{3-2i,2}(\Om))\cap L^2([0,T_0], W^{4-2i,2}(\Om))\,\, \text{for}\,\, i=0,1,\\
	\p^2_t h\in & L^2([0,T_0], L^2(\Om)).
\end{aligned}
\end{equation}
\end{thm}

\medskip
\noindent {\it\bf{Acknowledgements}}: The author B. Chen is partially supported by NSFC (Grant No. 12301074) and Guangdong Basic and Applied Basic Research Foundation (Grant No. 2025 A1515010502), and the author Y.D. Wang is partially supported by NSFC (Grant No. 12431003).

\medskip
\section*{Statements and Declarations}
\noindent {\it\bf{Conflict of interest}}: The authors declare that they have no conflict of interest.

\medskip
\noindent {\it\bf{Data Availability}}: Data sharing is not applicable to this article as no datasets were generated or analyzed during the study.

\end{document}